\documentclass[a4paper,11pt,DIV=11,abstract=on]{scrartcl}
\usepackage{graphicx}

\newdimen\figrasterwd
\figrasterwd\textwidth

\usepackage{mathtools}
\usepackage{authblk}
\mathtoolsset{showonlyrefs}
\usepackage{amsmath, amsthm, amssymb}
\usepackage{fontenc}
\usepackage[utf8]{inputenc}
\usepackage{graphicx}
\usepackage{subfig}
\usepackage{xargs}
\usepackage{enumitem,listings,microtype,tikz,pgfplots}
\usepackage{hyperref}
\usepackage{environ,ifthen}
\usepackage{multirow}
\usepackage{verbatim}

\newcommand{\R}{\mathbb{R}}
\newcommand{\N}{\mathbb{N}}

\newcommand{\tT}{\mathrm{T}}
\DeclareMathOperator*{\argmin}{arg\,min}
\DeclareMathOperator*{\argmax}{arg\,max}

\DeclareMathOperator{\midr}{midr}
\DeclareMathOperator{\Midr}{Midr}

\newcommand{\Knoten}{\mathcal{H}}

\newcommand{\Norm}[2]{\left\lVert#1\right\rVert_{#2}}

\newcommand\lex{\mathrm{lex}}
\newcommand\Llex{L\mbox{-}\mathrm{lex}}

\newtheorem{theorem}{Theorem}[section]

\newtheorem{proposition}[theorem]{Proposition}
\newtheorem{lemma}[theorem]{Lemma}
\newtheorem{remark}[theorem]{Remark}

\newtheorem{corollary}[theorem]{Corollary}
\setkomafont{sectioning}{\rmfamily\bfseries}
\setkomafont{title}{\rmfamily}

\def\l{\left}
\def\r{\right}

\begin{document}
\title{
Minimal Lipschitz and $\infty$-Harmonic Extensions \\
of Vector-Valued Functions on Finite Graphs}

\author[1]{Miroslav Ba{\v{c}}{\'a}k}
\author[2]{Johannes Hertrich}
\author[2]{Sebastian Neumayer}
\author[2,3]{Gabriele Steidl}
\affil[1]{\footnotesize Max Planck Institute for Mathematics in the Sciences, Inselstr.~22, 04103 Leipzig, Germany,
	bacak@mis.mpg.de}
\affil[2]{\footnotesize Department of Mathematics,
	Technische Universit\"at Kaiserslautern,
	Paul-Ehrlich-Str.~31, 67663 Kaiserslautern, Germany, \{jhetric,neumayer,steidl\}@mathematik.uni-kl.de}
\affil[3]{Fraunhofer ITWM, Fraunhofer-Platz 1,
	67663 Kaiserslautern, Germany}
\maketitle
	
\begin{abstract}
This paper deals with extensions of vector-valued functions on finite graphs fulfilling
distinguished minimality properties.
We show that so-called $\lex$ and $\Llex$ minimal extensions are actually the same  
and call them minimal Lipschitz extensions.
Then we  prove that the solution of the graph $p$-Laplacians
converge to these extensions as $p\to \infty$.
Furthermore, we examine the relation between minimal Lipschitz extensions and iterated
weighted midrange filters and address their connection  
to $\infty$-Laplacians for scalar-valued functions.
A convergence proof for an iterative algorithm proposed by Elmoataz \emph{et al.}~(2014) for finding the zero
of the $\infty$-Laplacian is given.
Finally, we present applications in image inpainting.
\end{abstract}


\section{Introduction} \label{sec:intro}
In this paper, we study an $\infty$-harmonic variant of the Dirichlet problem for vector-valued functions defined on finite graphs. 
To be more specific, we assume that a vector-valued function is given on a subset of the vertex set (``boundary'') 
and our goal is to extend this function onto the whole vertex set so that the extension is $\infty$-harmonic. 
However, as observed by \cite{SS2012}, $\infty$-harmonic extensions of vector-valued functions 
on finite graphs are not uniquely determined. To overcome this non-uniqueness issue they introduced 
the stronger notion of \emph{tight extensions} and
proved that such extensions exist and are unique. 
In the present paper, we approach these extensions from a different point of view 
and call them  \emph{minimal Lipschitz extension}, 
since its local Lipschitz constant (``oscillation'') is in some sense optimal. 
We therefore aim at finding a solution to the above boundary problem which is minimal Lipschitz.

In \cite[Section 2]{SS2012} an algorithm for computing these minimal Lipschitz extensions in the \emph{scalar-valued} case was presented. 
To the best of our knowledge there exists no meaningful algorithm for computing minimal Lipschitz extensions of \emph{vector-valued} functions 
and it is our goal to address their approximation in the present paper. Along the way we also clarify various aspects 
concerning minimal Lipschitz extensions and their relations to other concepts including iterated midrange filters, minimizers of $p$-energy functionals, $p$-Laplacians 
and $\infty$-Laplacians.

The topic of minimal Lipschitz extensions on graphs
appears in various subfields of mathematics and computer science from different
points of view and with different notations.
These areas include approximation theory \cite{Descloux1963}, 
discrete mathematics/graph theory \cite{PSSW2005}, 
data and image processing \cite{BN2003,elion-vese,ETT2015,SM2000} including manifold-valued data \cite{BT2018}, and 
mathematical morphology \cite{Serra1982} to mention only a few.
While many results on minimal Lipschitz extensions are available for scalar-valued functions, the vector-valued case has been less studied.

We work with \emph{weighted} graphs which turns out to be crucial for the applications in image processing. 
Indeed, in Section \ref{sec:numerics}, we present applications of minimal Lipschitz extensions for inpainting of vector-valued images
which rely upon representing the image under consideration as a function on an appropriate weighted graph which is obtained by \emph{nonlocal} patch-based
techniques.

As a matter of fact, \cite{SS2012} define tight extensions also for functions defined 
on a bounded, open, connected subset $U\subset\R^n$. 
This is closely related to \emph{absolutely minimizing Lipschitz extensions} (AMLE)
and therefore also to $\infty$-harmonic functions. 
Indeed, \cite{Jensen1993}, stimulated by the work \cite{aronsson1967}, proved that for
$g\in C(\partial U; \mathbb R)$,
the  boundary-value problem
\begin{align*}
-\Delta_\infty f  = 0 \quad \text{in } U, \qquad f =g \quad \text{on } \partial U,
\end{align*}
with the $\infty$-\emph{Laplacian} defined by
\begin{equation*}
\Delta_\infty f \coloneqq |\nabla f|^{-2}\sum_{i,j=1}^n \partial_i f \; \partial_{ij} f \; \partial_{j} f,
\end{equation*}
has a unique viscosity solution $f\in C(\bar{U};\mathbb R)$.
\cite{CEG} showed its equivalence to the AMLE of $g$.
In connection with image interpolation algorithms and elliptic partial differential operators, AMLEs 
were studied by \cite{CMS98} inspired by the work of \cite{CT1996}.
The operator was considered, e.g.,~for comparisons of image compression algorithms in \cite{GWWBBS2008,SPMEWB2014}.
Even though we are aware of the existence of this deep theory of Lipschitz extensions in $\R^n,$ 
we focus exclusively on functions defined on finite graphs in the present paper. 
For more information on the continuous case, 
the interested reader is also referred to the recent papers \cite{katzourakis2012,katzourakis2014,lindqvist2016} and the references therein.

This paper is organized as follows. 
After fixing the notation in Section \ref{sec:discr_lap}, we introduce $\lex$ and $\Llex$ minimal extensions 
in Section \ref{sec:p_lap} and prove that they actually coincide. 
In Section \ref{approx}, we show that the minimizers of the grouped $p$-energy functionals $E_p$ 
converge to these extensions as $p\to \infty$.
Section \ref{sec:midrange+harm} deals with the relation between iterated midrange filters and minimal Lipschitz extensions.
In Section \ref{sec:inf_lap},  we consider $\infty$-Laplacians for scalar-valued functions and provide a convergence proof for an iterative algorithm of \cite{EDLL2014}. 
Section~\ref{sec:algs} contains finer analyses of the numerical algorithms. 
Finally, Section \ref{sec:numerics} shows applications in image inpainting.
Preliminary results of the present paper are contained in the conference paper \cite{HNSB2019}.
\section{Preliminaries} \label{sec:discr_lap}
%
Let $G \coloneqq (V,E,w)$ be a finite, undirected, connected, weighted graph 
with weight function $w \colon E\to [0,1].$ We use the usual notation $u \sim v$ for $(u,v) \in E.$ Let $\emptyset \neq U \subset V$ and assume $u\not\sim v$ 
if $u,v \in U$.
Denote $M\coloneqq |E|$ and $N\coloneqq|V \setminus U|$. 
We also suppose that $w(u,v) > 0$ if and only if $u\sim v.$ 
Since the graph is not directed, the weights are symmetric, that is, $w(u,v) = w(v,u)$. 
Finally, we suppose that $w(u,u) > 0$ for $u \in V \setminus U$.

The set of functions $f\colon V \to \R^m$ is denoted by $\Knoten(V)$.
For a given function $g\colon U \to \R^m$, 
let $\Knoten_g(V)$ denote those functions  $f \in \Knoten(V)$ with $f\restriction_U = g$, 
which are called \emph{extensions of} $g$.

The \emph{2-$p$-norm}, for $p \in [1,\infty),$ and the \emph{$2$-$\infty$-norm} are defined for $x=\left(x_1,\dots,x_n\right)\in \R^{mn}$ as
$
\Norm{x}{2,p}\coloneqq \bigl(\sum_{i=1}^n  |x_i |^p \bigr)^{1/p} 
$, and
$
\Norm{x}{2,\infty}\coloneqq \max_{i=1,\dots,n}  |x_i |
$,
where $|x_i|$ denotes the Euclidean norm of $x_i \in \R^m$.
For $m=1$, we have just the usual $\ell_p$-norms which we denote by $\| \cdot\|_p$, for $p \in [1,\infty]$.

In \cite{GO2008} the discrete gradient operator $\nabla_w\colon \Knoten(V) \to \R^{2mM}$ was introduced
by
$
\nabla_w f(u) \coloneqq ( \partial_v f(u) )_{v \sim u},
$
where
$
\partial_v f(u) \coloneqq w(u,v) ( f(v) - f(u)) \in \R^{m}.
$
Let $g\colon U \to \R^m$ be given.
For $f \in \Knoten_g(V)$, we are interested in the \emph{anisotropic energies of the $p$-Laplacians}
\begin{align} \label{E_p}
E_{p}f
&\coloneqq     \| \nabla_w f  \|_{2,p} ^p 
=  \sum_{u \in V} \Big( \sum_{v \sim u} w(u,v)^p |f(u) - f(v)|^p \Big), \\
E_{\infty} f
&\coloneqq  \| \nabla_w f  \|_{2,\infty} 
= \max_{u \in V} \left(\max_{v \sim u} w(u,v) |f(u) - f(v)| \right)\\
&\, = \max_{u \in V\setminus U} \left(\max_{v \sim u} w(u,v) |f(u) - f(v)| \right).
\end{align}
The functionals $E_p$, for $p \in (1,\infty),$ are strictly convex and hence they have a unique global minimum $f_p$.
Besides $E_p$, the functional
\begin{align*}
E_{\infty,p} f
&\coloneqq \left\| \, \left( \| \nabla_w f (u) \|_{2,\infty} \right)_{u \in V \setminus U} \, \right\|_p^p 
=  
\sum_{u \in V\setminus U} \Big( \max_{v \sim u} w(u,v)^p |f(u) - f(v)|^p \Big), \quad p \in [1,\infty)
\end{align*}
was considered by \cite{SS2012}.
This functional is not strictly convex for $p \in (1,\infty)$, 
but has nevertheless a unique minimizer $f_{\infty,p} \in \Knoten_g(V)$ of $E_{\infty,p}$;
see \cite{Hertrich2018}.
In contrast to $E_p$ or $E_{\infty,p}$, for $p \in (1,\infty)$, the functional $E_{\infty}$ has in general many minimizers.
In this paper, we want to accent minimizers of $E_\infty$ with distinguished properties.

Using $\Gamma$-convergence arguments, it is not hard to show that
every cluster point of the sequence of the minimizers of $E_p$, resp.~$E_{\infty,p}$
is a minimizer of $E_\infty$. For convenience, we add the short proof.

\begin{lemma} \label{lem:gamma}
	Every cluster point of the sequence of minimizers $\{ f_p\}_p$ of $E_p$ 
	is a minimizer of $E_\infty$. The same holds true for the minimizers $\{ f_{\infty,p} \}_p$ of $E_{\infty,p}.$ 
\end{lemma}
\begin{proof} Recall that the $\ell_p$-norms satisfy
	\begin{equation*}
	\|x\|_\infty \le \|x\|_p \le N^{1/p} \|x\|_\infty, \quad x \in \R^N.
	\end{equation*}
	In particular, $\| \cdot \|_p \to \| \cdot \|_\infty$ as $p \to \infty$
	uniformly on bounded sets and the convergence is monotone.
	These properties are inherited by the convergence $E_p^{1/p} \to E_\infty$ as $p \to \infty$.
	This implies that $E_p^{1/p}$ $\Gamma$-converges to $E_\infty$ as $p \to \infty$ and we obtain the desired results since the functionals $E_p^{1/p}$ and $E_p$ have the same minimizers.
	Similar arguments can be applied for the functional $E_{\infty,p}$. 
\end{proof}
\section{Minimal Lipschitz Extensions} \label{sec:p_lap}

Let again $U \subset V$ be a nonempty subset of the vertices of the graph $G \coloneqq (V,E,w)$ and let $g\colon U\to \R^m.$ 
We start by recalling the definitions of two types of minimal extensions of $g$ and we then show in Theorem \ref{thm:coincide} below that they actually coincide.
For $u \in V \setminus U$, let 
\begin{equation*}
lf(u,v) \coloneqq w(u,v)|f(u) - f(v)|= |\partial_v f(u)|
\end{equation*}
and 
$\lex(f) \in \R^M$ be the vector
with entries $\left( lf(u,v) \right)_{v \sim u}$ in nonincreasing order.
Note that we  count the entries $lf(u,v) = lf(v,u)$ only once.
For $u \in V \setminus U$, define
\begin{equation*}
Lf(u) \coloneqq \max_{v \sim u} lf(u,v) =  \| \nabla_w f  (u)\|_{2,\infty},
\end{equation*}
and the vector
$\Llex (f) \in \R^{|V\setminus U|}$ with entries $\left(Lf(u) \right)_{u\in V \setminus U}$ in nonincreasing order.
Denoting by $\le$ the lexicographical ordering, a function $f\in \Knoten_g(V)$ is called 
\begin{enumerate}
	\item
	$\lex$ \emph{minimal extension} (of $g$) if 
	$\lex(f) \le \lex(h)$ for every $h \in \Knoten_g(V)$, and
	\item
	$\Llex$ \emph{minimal extension} (of $g$) if $\Llex (f) \le \Llex(h)$ for every $h \in \Knoten_g(V)$.
\end{enumerate}
The first notation can be found for instance in~\cite{KRSS2015}.
The $\Llex$ \emph{minimal extension} was called \emph{tight extension} in the paper \cite{SS2012}.

The existence of $\Llex$ minimal extensions in the non-weighted case was shown in \cite[Theorem 1.2]{SS2012}.
The existence in the weighted setting as well as for $\lex$ minimal extensions can be proved similarly; see \cite{Hertrich2018}.
By definition we clearly have $\lex_1(f) = \Llex_1(f)$, where the subscript denotes the coordinate index.
In this section, we want to show that $\lex$ and $\Llex$ minimal extensions indeed coincide.
To this end, we need the following lemma.

\begin{lemma} \label{lem:wichtig}
	Let $f \in \mathcal{H}_g(V)$ be a $\lex$ minimal extension and $\tilde f \in  \mathcal{H}_g(V)$
	with $f \neq \tilde f$ and $\lex_1 (f) = \lex_1 (\tilde f)$.
	Let $k \in \{1,\ldots, M\}$ be the largest index such that
	$( \lex_i (f) )_{i=1}^k = ( \lex_i (\tilde f) )_{i=1}^k$
	and $K \coloneqq \lex_k (f)$.
	Then $f$ and $\tilde f$ coincide on the set
	\[
	W \coloneqq \left\{ u \in V\setminus U: \exists v \sim u \quad \mbox{such that} \quad lf(u,v) \ge K \right\}.
	\]
\end{lemma}
Note that for a $\lex$ minimal extension $f\in \mathcal{H}_g(V)$ the case $k=M$, i.e.~$\lex (f) = \lex (\tilde f)$ implies $f = \tilde f$.
This means that the $\lex$ minimal extension is unique.
\begin{proof}
	1. First, we prove that $f(u) - f(v) = \tilde f(u) - \tilde f(v)$
	for all $u \sim v$ with $lf(u,v) \ge  K$. Consider $u \sim v$ with $lf(u,v) \ge  K$.
	We suppose that if
	$lf(u,v) =  lf(\tilde u, \tilde v) = l\tilde f(u,v) = l\tilde f(\tilde u,\tilde v)$
	and those values appear at positions $i$ and $j$ in $\lex(f)$ and $\lex(\tilde f)$, then
	the corresponding values $lf(u,v), l\tilde f(u,v)$ and $lf(\tilde u,\tilde v), l\tilde f(\tilde u,\tilde v)$ have the same position.
	
	For a contradiction, assume it is not the case and let thus $C$ be the largest value in $( \lex_i (f))_{i=1}^k$
	with $f(u) - f(v) \neq \tilde f(u) - \tilde f(v)$.
	For $u \sim v$ with $lf(u,v) > C$, 
	we have $f(u) - f(v) = \tilde f(u) - \tilde f(v)$
	and for $lf(u,v) = C$ the relation $|f(u) - f(v)| \ge |\tilde f(u) - \tilde f(v)|$,
	where at least one $u \sim v$ with $f(u) - f(v) \neq \tilde f(u) - \tilde f(v)$ exists.
	For $u \sim v$ with $lf(u,v) < C$ we have $l\tilde f (u,v) \le C$.
	For $h \coloneqq \frac12 (f + \tilde f) \in \Knoten_g(V)$ and  $u \sim v$, we obtain 
	\begin{equation}
	|h(u) - h(v)| =  \tfrac12 \bigl|f(u) - f(v) + \tilde f(u) - \tilde f(v)\bigr|
	\le  \tfrac12  \left( |f(u) - f(v)| +  |\tilde f(u) - \tilde f(v)|\right)\label{ineq}.
	\end{equation}
	Consequently, by the previous considerations,
	$
	lh(u,v) \le lf(u,v)
	$ 
	whenever $lf(u,v) \ge C$. Consider $u\sim v$ with  $lf(u,v) = C$  with $f(u) - f(v) \neq \tilde f(u) - \tilde f(v)$.
	Then either
	$|f(u) - f(v)| > |\tilde f(u) - \tilde f(v)|$ which implies
	$
	lh(u,v) \le lf(u,v)
	$ 
	or $|f(u) - f(v)| =  |\tilde f(u) - \tilde f(v)|$.
	For two vectors $a,b \in  \R^m$ with $|a| = |b|$ we have $|a+b| = |a| + |b|$ if and only if $a = b$.
	Thus we have strict inequality in \eqref{ineq} which results again in the strict inequality $lh(u,v) < lf(u,v)$.
	Finally we conclude $\lex(h) < \lex (f)$  which contradicts the $\lex$ minimality of $f$.
	Hence  $f(u) - f(v) = \tilde f(u) - \tilde f(v)$ for all $u \sim v$ with $lf(u,v) \ge \lex_k(f)$.
	
	2. Next, we show that there exits $u \sim v$ with $v \in U$ such that
	$lf(u,v) = \lex_1(f)$.
	Assume in contrary that this is not the case, i.e., there exists $\delta >0$ such that 
	$lf(u,v) \le \lex_1(f) - \delta$ for all $u \sim v$ with $v \in U$.
	For $\varepsilon > 0$, consider the function $\hat f \in \mathcal{H}_g(V)$ with
	$\hat f (u) \coloneqq (1- \varepsilon) f(u)$, $u \in V\setminus U$.
	Then we have 
	for all $u_1,u_2 \in  V\setminus U$ with $u_1 \sim u_2$ that
	$l\hat f(u_1,u_2) = (1-\varepsilon) lf(u_1,u_2)  < lf(u_1,u_2)$
	and 
	for all $u \in V\setminus U$ and $v \in U$ that
	\begin{align*}
	l \hat f(u,v) 
	= w(u,v) |(1-\varepsilon) f(u) - f(v)| \le lf(u,v) + \varepsilon w(u,v) |f(u)|
	< \lex_1 f
	\end{align*}
	for $\varepsilon < \delta/ \max_{u\sim v} \left( w(u,v) |f(u)| \right)$.
	Thus $\lex \hat f < \lex f$ which is a contradiction.
	
	3. Let $\tilde f \in \mathcal{H}_g(V)$ with $\lex_1(f) =\lex_1(\tilde f)$. 
	For $u \sim v$ with $v \in U$ such that $lf(u,v) = \lex_1(f)$, Part 1 of the proof implies that $f(u) = \tilde f(u)$.
	Choose one such $u$ and set $U_1 \coloneqq U \cup \{u\}$, extend $g$ to $g_1$ on $U_1$ by  
	$g_1(u) \coloneqq f(u)$.  
	Cut all edges $u\sim v$ with $v \in U$ and remove the corresponding entries in  $\lex(f)$ and $\lex(\tilde f)$.
	Note that only entries with the same value are removed, including the first one.
	Consider $f$ and $\tilde f$ as extensions of  $g_1$, where $f$ is still the $\lex$ minimal extension of $g_1$.
	
	The whole procedure is repeated with respect to the new first component $\lex_1(f)$ and so on
	until all edges with $lf(u,v)\ge K$ are removed. This yields the assertion. 
\end{proof}

Now we can prove the equivalence between $\lex$ and $\Llex$ minimal extensions.
\begin{theorem}\label{thm:coincide}
	There exists a unique $\lex/ \Llex$ minimal  extension $f \in \Knoten_g(V)$
	and both extensions coincide.
\end{theorem}
\begin{proof}
	1. The uniqueness of the $\Llex$ minimal extension in the unweighted case was shown in \cite[Theorem 1.2]{SS2012}
	and it is not hard to extend the arguments to the weighted setting; see \cite{Hertrich2018}.
	The uniqueness of the $\lex$ minimal extension follows directly from Lemma \ref{lem:wichtig}.
	
	2. 
	Let $f\in \Knoten_g(V)$ be the $\lex$ minimal extension and $\tilde f\in \Knoten_g(V)$ the $\Llex$ minimal extension.
	Assume that $f \neq \tilde f$.
	Then $\lex (f) < \lex (\tilde f)$ and $\Llex (f) > \Llex (\tilde f)$.
	If
	$\Llex_1(f) = \lex_1 (f) < \lex_1 (\tilde f) = \Llex_1(\tilde f)$ we have a contradiction.
	Thus, $\lex_1 (f) = \lex_1 (\tilde f)$ and we can apply Lemma \ref{lem:wichtig}
	which implies that $f$ and $\tilde f$ coincide on the set $W$.
	Now we can consider $f$ and $\tilde f$ as extensions of $g$ extended to $U \cup W$,
	where the edges between vertices in $U \cup W$ and the corresponding entries in $\lex (f)$ and $\lex(\tilde f)$ are removed.
	But for this new constellation 
	we have $\lex_1 (f) < \lex_1 (\tilde f)$ which again contradicts the
	$\Llex$ minimality of~$\tilde f$. 
\end{proof}

We call the minimal $\lex/ \Llex$ extension of $g$ on $G$  the \emph{minimal Lipschitz extension} of $g$, or sometimes shortly the \emph{minimal extension} of $g$.

\section{Approximations of Minimal Lipschitz Extensions}\label{approx}

In this section, we focus on algorithms for computing minimal Lipschitz extensions. 
In the scalar-valued case minimal Lipschitz extensions coincide with $\infty$-harmonic functions 
and one can therefore use algorithms for $\infty$-harmonic functions. Such an algorithm (for non-weighted graphs) 
can be found in \cite[Section 2.1]{SS2012}, where the authors reference earlier papers \cite{LLPSU1999,PSSW2005}. 
Note that it is a polynomial-time algorithm. For weighted graphs, a polynomial-time algorithm for computing $\infty$-harmonic functions was given in \cite{KRSS2015} 
along with a faster variant for practical computations.

To the best of our knowledge there exist no methods for computing minimal Lipschitz extensions of general vector-valued functions; this problem is also implicitly formulated in \cite[Section 2.1]{SS2012}. 
Our aim is to address this issue. To this end recall the following theorem, which is just a simple generalization of a result in \cite[Theorem 1.3]{SS2012} for weighted graphs.
\begin{theorem}
	For any $m \ge 1$, the sequence $\{ f_{\infty,p} \}_{p}$ of minimizers $f_{\infty,p} \in \Knoten_g(V)$ of $E_{\infty,p}$ 
	converges to the minimal Lipschitz extension of $g,$ as $p\to\infty.$
\end{theorem}
\begin{proof}
	See \cite{Hertrich2018} for the weighted-graph case and \cite[Theorem 1.3]{SS2012} for the non-weighted graph case.
\end{proof}

\begin{remark}
	Let us mention that instead of considering $E_{\infty,p}$ we can also prove convergence for the minimizers
	of other functionals
	\begin{equation*}
	E_{\infty,\varphi_p} f \coloneqq \varphi_p \left(\| \nabla_w f\|_{2,\infty} \right),
	\end{equation*}
	where $\varphi_p$ has to fulfill certain properties. For example, the function 
	\begin{equation*}
	\varphi_p(x) \coloneqq \frac1p \mathrm{logexp}(p x) = \frac1p \log \big( \sum_{i=1}^N \exp(px_i) \big)
	\end{equation*}
	is an appropriate choice since the asymptotic function of the function $\operatorname{logexp}$ is the $\operatorname{vecmax}$ function, that is,
	$\lim_{p \to \infty} \frac1p \log \left( \sum_{i=1}^N  \exp(px_i) \right)= \max_{i=1,\ldots,N} x_i$.
\end{remark}
We proceed by studying the sequence $\{f_p\}_p$ of minimizers of $E_p$. For scalar-valued functions, that is, $m=1$, 
its convergence to the minimal extension of $g$ can be shown
by applying the following classical result from approximation theory 
concerning the convergence of P\'olya's algorithm \cite{Polya1913}:
In \cite{Descloux1963} it was proved that, given an affine subspace ${\mathcal K} \subset \R^M$ and $z \in \R^M$, the sequence of $L_p$ approximations
$\{x_p\}_p$ 
defined by
\[x_p \coloneqq \argmin_{x \in {\mathcal K}} \|x - z\|_p, \quad p \in (1,\infty), \]
converges (as $p\to\infty$) to the minimizer
$\hat x_\infty\in\mathcal{K}$ of $\|x - z\|_\infty$
with the following property:
For every  minimizer $x_\infty\in\mathcal{K}$ of $\Vert x - z \Vert_\infty$
consider the vector $\sigma(x_\infty)$ whose coordinates $|x_{\infty,i} - z_i|$ are arranged in nonincreasing order.
Then $\hat x_\infty$ is the minimizer with smallest lexicographical ordering with respect to $\sigma$.
Indeed this vector is uniquely determined and called \emph{strict uniform approximation} of $z$.
Concerning the convergence rate it was shown in \cite{QFMB2005}, see also \cite{EH1990}, that
there exist constants $0< C_1,C_2 < \infty$ and $a \in [0,1]$ depending on ${\mathcal K}$,
such that
\begin{equation*}
C_1 \, \tfrac{1}{p} \, a^p  \le  \|x_p - \hat x_\infty \|_\infty \le C_2 \, \tfrac{1}{p} \, a^p.
\end{equation*}
To apply this result in to minimal Lipschitz extensions, we rewrite
$
E_p f =  2\, \Vert Af\restriction_{V \setminus U} \, + \, b \Vert_p^p,
$
where $A \in \R^{M,|V \setminus U|}$ 
is the matrix representing the linear operator 
$\nabla_w$ on functions restricted to $V \setminus U$ and 
$b \in \R^{M}$ accounts for the fixed values $g$ on $U$.
Then, considering ${\mathcal K} \coloneqq \{Ay + b\colon y \in \R^{|V \setminus U|}\}$
and $z\coloneqq 0$, we obtain the desired result.

For the vector-valued case, that is $m \ge 2$, the convergence of the minimizers of $E_p$ cannot be
deduced in this way and we therefore prove the following theorem.
\begin{theorem}\label{thm:ConvEp}
	For any $m \ge 1$, the sequence $\{ f_{p} \}_{p}$ of minimizers $f_{p} \in \Knoten_g(V)$ of $E_p$ 
	converges to the minimal Lipschitz extension of $g,$ as $p\to\infty.$
\end{theorem}
\begin{proof}
	The sequence $\{f_{p}\}_p$ is bounded and hence there exists some accumulation point.
	Let $f \in \Knoten_g(V)$ be the minimal extension and assume that there exists a convergent subsequence $\{f_{p_j}\}_j$ of $\{f_{p}\}_p$ with limit $\tilde f \neq f$.
	Choose $k \in \{1,\ldots, M\}$ to be the largest index such that
	$( \lex_i (f) )_{i=1}^k = ( \lex_i (\tilde f))_{i=1}^k$, which exists since $\tilde f$ is a minimizer of $E_\infty$.
	By Lemma \ref{lem:wichtig} the functions $f$ and $\tilde f$ coincide on the set
	$
	W \coloneqq \left\{ u \in V\setminus U\colon \exists v \sim u \quad \text{such that} \quad lf(u,v) \ge \lex_k (f) \right\}.
	$
	Define
	\begin{equation}
	h_p(u)\coloneqq 
	\begin{cases}
	f_p(u) & \mathrm{if} \;  u\in W\cup U,\\
	f(u) &\mathrm{otherwise.}
	\end{cases}
	\end{equation}
	Since $f(u)=\tilde f(u)$ for $u\in W$, we have that $h_{p_j}\to f$ as $j\to \infty$.
	For $u\in W\cup U$ and $v\in W$ with $u\sim v$ it holds 
	\begin{equation}
	w(u,v) |f_{p_j}(u)-f_{p_j}(v)| = w(u,v)|h_{p_j}(u)-h_{p_j}(v)|,
	\end{equation}
	and for $u\in V$ and $v\in V\setminus (W\cup U)$ with $u\sim v$,
	\begin{equation}
	w(u,v) | h_{p_j}(u)-h_{p_j}(v) | \to w(u,v) |f(u)-f(v)| \leq \lex_{k+1}(f).
	\end{equation}
	Setting $\delta\coloneqq \lex_{k+1} (\tilde f)-\lex_{k+1} (f) >0$,
	we obtain, for sufficiently large $j$ and $u\in V$, $v\in V\setminus (W\cup U)$ with $u\sim v,$ that
	\begin{equation}
	w(u,v)|h_{p_j}(u)-h_{p_j}(v)| \leq \lex_{k+1}(f) + \frac\delta4.
	\end{equation}
	On the other hand, there exist $u\in V$ and $v\in V\setminus (W\cup U)$ with $u\sim v$ such that 
	\begin{equation}
	w(u,v)|f_{p_j}(u)-f_{p_j}(v)|\to w(u,v) |\tilde f(u)-\tilde f(v)| = \lex_{k+1}(\tilde f)
	\end{equation}
	so that, for sufficiently large $j$,
	\begin{equation}
	w(u,v)|f_{p_j}(u)-f_{p_j}(v)| \geq \lex_{k+1}(\tilde f) - \frac\delta4.
	\end{equation}
	Then we conclude
	\begin{align*}
	E_{p_j}(f_{p_j})-E_{p_j}(h_{p_j})
	& = \sum_{u \in V} \sum_{v\sim u} w(u,v)^{p_j} \big( |f_{p_j}(v)-f_{p_j}(u) |^{p_j} - | h_{p_j}(v)-h_{p_j}(u)|^{p_j} \big)
	\\
	& = \sum_{u\in W\cup U}\sum_{\substack{v\sim u\\v\not\in (W\cup U)}}w(u,v)^{p_j} \left|f_{p_j}(v)-f_{p_j}(u)\right|^{p_j}\\
	& \quad +\sum_{u\in V\setminus (W\cup U)}\sum_{v\sim u} w(u,v)^{p_j} \left|f_{p_j}(v)-f_{p_j}(u)\right|^{p_j}\\
	& \quad -\sum_{u\in W\cup U}\sum_{\substack{v\sim u\\v\notin (W\cup U)}}w(u,v)^{p_j} \left|h_{p_j}(v)-h_{p_j}(u)\right|^{p_j}\\
	& \quad-\sum_{u\in V\setminus (W\cup U)}\sum_{v\sim u} w(u,v)^{p_j} \left|h_{p_j}(v)-h_{p_j}(u)\right|^{p_j},
	\end{align*}
	and further,
	\begin{align*}
	E_{p_j}(f_{p_j})-E_{p_j}(h_{p_j})
	\geq
	& \left(\lex_{k+1}(\tilde f) - \frac\delta4 \right)^{p_j} - C \left(\lex_{k+1}(f) + \frac\delta4 \right)^{p_j}\\
	=&
	\left( \lex_{k+1}(\tilde f) - \frac\delta4 \right)^{p_j} 
	\left(1- C \left( \frac{ \lex_{k+1}(f) + \frac\delta4 }{ \lex_{k+1}(\tilde f) - \frac\delta4 } \right) ^{p_j} \right).
	\end{align*}
	By the definition of $\delta,$ the quotient is smaller than $1$ and we thus obtain, for sufficiently large $j,$ 
	that $E_{p_j}(f_{p_j})-E_{p_j}(h_{p_j}) > 0$.
	This contradicts that $f_{p_j}$ is the minimizer of $E_{p_j}$ and the proof is complete. 
\end{proof}

\section{Midrange Filters and $\infty$-Harmonic Extensions} \label{sec:midrange+harm}
This section relates midrange filters, $\infty$-harmonic extensions and minimal Lipschitz extensions. 
We then arrive at the Krasnoselskii--Mann algorithm which provides us with an approximation algorithm for minimal Lipschitz extensions. 
Let us start by recalling the concept of midrange filters.

\subsection{Midrange Filters}\label{subsec:midr}
Consider the following minimization problem:
For $x = (x_1,\ldots,x_n)^\tT$, $x_i \in \R^m$ 
and $w = (w_1,\ldots,w_n)^\tT \in (0,1]^n$ with $w_1 \ge \ldots \ge w_n$,
define the \emph{weighted midrange filter} $\midr_w\colon \R^{mN} \to \R^m$ given by
\begin{align}\label{eq:Midrange}
\midr_w x
&= \argmin_{a \in \R^m} \big\{ \max_i \left( w_i |x_i - a| \right) \big\}. 
\end{align}
If $w_i |x_i - a|$ is replaced by $w_i^2 |x_i - a|^2$, $i=1,\ldots,n$, the minimizer remains the same. 
As the pointwise maximum of strongly convex functions with modulus $2 w_i^2$,
the  functional $\max_i \left( w_i^2 |x_i - a|^2 \right)$ is strongly convex with modulus $2 w_n^2$.
Hence $\midr_w x$ is uniquely determined. 

In the scalar-valued case, that is $m=1$, in was shown in \cite[Theorem~5]{Obe05} that the midrange filter can be expressed as
\begin{equation}\label{eq:ScalarMidr}
\midr_w x =  \frac{w_i x_i + w_j x_j}{w_i + w_j} , \qquad
(i,j) \in \argmax_{k,l} \frac{\vert x_k - x_l\vert}{1/w_k + 1/w_l},
\end{equation}
which further simplifies to
\begin{equation*}
\midr_w x =  \tfrac12 \Bigl(\min_i x_i + \max_i x_i\Bigr),
\end{equation*}
for the non-weighted filter.

Note that we call a midrange filter from the signal processing point of view, 
is  quite classic and can be found in the literature under various names including the \emph{smallest circle/bounding sphere center,} or \emph{Chebyshev center,} or \emph{circumcenter}.
The problem is known to be solvable by an $\mathcal{O}(n)$ linear programming algorithm, where the factor in the $\mathcal{O}(n)$ term depends sub-exponentially 
on the dimension $m;$ see~\cite{MSW1996}.
For a comparison of several algorithms, see~\cite{XuFrSu03}.

The minimization problem \eqref{eq:Midrange} can be generalized from $\R^m$ into Hadamard spaces (see \cite[Example 2.2.18]{B2014}), which is of interest when 
one works, for instance, with the Hadamard manifold of symmetric positive definite matrices; see~\cite{WWBSFB07}.
\begin{lemma}\label{lem:lip}
	For scalar-valued functions, that is, $m=1$, the operator $\midr_{w}$ is Lipschitz continuous
	with constant $L \le 1$ with respect to the $\infty$-norm.
\end{lemma}
\begin{proof}
	Given $x,y \in \R^n,$ consider the line segment $x(t) \coloneqq (1-t) x + t y$, where $t \in [0,1]$. 
	Then the functions $L_{i,j}\colon [0,1] \to \R$, where $i,j =1,\ldots,n,$ defined by 
	\[L_{i,j}(t) = \frac{\vert x_i(t) - x_j(t) \vert}{1/w_i + 1/w_j}\]
	are piecewise linear and hence 
	$L(t) \coloneqq \sup_{i,j} L_{i,j}(t)$ is also piecewise linear. 
	Therefore, $[0,1]$ can be split into a finite number of intervals 
	$[t_{k-1},t_k]$, $k=1,\dots,K$ with $t_0 \coloneqq 0$ and $t_K \coloneqq 1$ 
	with corresponding maximizing indices $(i_{k}, j_{k})$ on $[t_{k-1},t_k]$, $k=1,\ldots,K$. 
	By the triangle inequality and \eqref{eq:ScalarMidr}, we get
	\begin{align}
	\bigl \vert \midr_{w} x - \midr_{w} y \bigr \vert
	&\leq \sum_{k=1}^K \bigl\vert \midr_{w} x(t_{k-1}) - \midr_{w} x(t_k) \bigr \vert\\
	&= 
	\sum_{k=1}^n \left\vert \frac{w_{i_k} x_{i_k}(t_{k-1})+ w_{j_k} x_{j_k}(t_{k-1})}{w_{i_k} + w_{j_k}}
	- \frac{ w_{i_k} x_{i_k}(t_{k}) + w_{j_k} x_{j_k}(t_{k})}{w_{i_k} + w_{j_k}}\right\vert\\
	&= 
	\sum_{k=1}^n \biggl\vert \frac{w_{i_k} \bigl(x_{i_k}(t_{k-1})-x_{i_k}(t_{k})\bigr) 
		+ w_{j_k} \bigl(x_{j_k}(t_{k-1}) - x_{j_k}(t_{k})\bigr)}{w_{i_k} + w_{j_k}}\biggr\vert.
	\end{align}
	Finally, the definition of $x(t)$ implies
	\begin{equation*}
	\bigl \vert \midr_{w} x - \midr_{w} y \bigr \vert
	\leq \sum_{k=1}^n \frac{w_{i_k}(t_{k} - t_{k-1})}{w_{i_k} + w_{j_k}} 
	\Vert x - y \Vert_\infty + \frac{w_{j_k}(t_{k} - t_{k-1})}{w_{i_k} + w_{j_k}} \Vert x - y \Vert_\infty\\
	= \Vert x - y \Vert_\infty,
	\end{equation*}
	which concludes the proof.
\end{proof}

Unfortunately, the operator $\midr_{w}$ is not Lipschitz in dimensions $m \ge 2$. As a matter of fact, one can find a counterexample already in the non-weighted case for $m=2$ and $n=3.$
In general we can only show that $\midr_{w}$ is locally $\frac{1}{2}$-H\"older continuous.
\begin{lemma} \label{lem:hoelder}
	The operator $\midr_{w}$ is locally $\frac{1}{2}$-H\"older continuous.
\end{lemma}
\begin{proof}
	Let $x,y \in \R^{nm}$ with distance $\|x - y\|_{2,\infty} \le r$ be given.
	Then $\hat x \coloneqq  \midr_w x$ and $\hat y \coloneqq  \midr_w y$ fulfill 
	$|\hat x - \hat y| \le C_r$
	with a constant $C_r$ depending only on $r$.
	Set 	
	$R(x,a) \coloneqq \max_i \left( w_i |x_i - a| \right)$. 
	Without loss of generality we can and do assume $R(x,\hat x) \ge R(y,\hat y)$. We then obtain
	\begin{equation}\label{ab1}
	R(x,\hat y)\leq R(y,\hat y) + w_{1}\Vert x - y\Vert_{2,\infty}
	\le
	R(x,\hat x) + w_{1}\Vert x - y\Vert_{2,\infty}. 
	\end{equation}
	Since the function $R(x,\cdot)^2$ is strongly convex with parameter $2w_{n}^2$ we have
	\begin{equation*}
	R(x,\hat x)^2 + w_{n}^2 |\hat x - \hat y|^2 \leq R(x,\hat y)^2,
	\end{equation*}
	which together with \eqref{ab1} this results in
	\begin{equation*}
	\Vert x - y\Vert_{2,\infty} \geq \frac{\sqrt{R(x,\hat x)^2 + w_{n}^2 |\hat x - \hat y|^2} 
		- R(x,\hat x)}{w_1} 
	= 
	\frac{w_{n}^2 \vert \hat x - \hat y\vert^2}{ w_{1}\bigl(\sqrt{R(x,\hat x)^2 + w_{n}^2 |\hat x - \hat y|^2} 
		+ R(x,\hat x)\bigr)},
	\end{equation*}
	and we arrive at
	\begin{equation*}
	|\hat x - \hat y|^2 \le \frac{w_{1}}{w_{n}^2} \, \bigl(\sqrt{R(x,\hat x)^2 + w_{n}^2 C_r^2} + R(x,\hat x)\bigr) \, \Vert x - y\Vert_{2,\infty}.
	\end{equation*}
	Hence, $\midr_{w}$ is locally $\frac{1}{2}$-H\"older continuous.
\end{proof}

\begin{remark}\label{rem:Lipschitz}
	For the non-weighted case, a detailed discussion on continuity properties of midrange filters can be found in \cite{IvaSos08}. 
	In particular, it was proved that for pairwise different $x_i \in \R^m$,
	there exists a ball around $x \in \R^{mn}$ such that the midrange operator is Lipschitz continuous on this ball.
\end{remark}
\subsection{$\infty$-Harmonic Extensions}
%
Next, we are interested in applying the midrange filter to given functions
$f\in \Knoten_g(V)$.
Define the operator
$\Midr_w\colon \Knoten_g(V) \to \Knoten_g(V)$  by
\begin{align}
\Midr_w f(u) 
&\coloneqq 
\midr_{w(u)} x(u) \\
&\, =\argmin_{a\in \R^m} \left\{ \max_{v\sim u} w(u,v) |f(v)-a| \right\}, \qquad u \in V \setminus U,
\end{align}
where
$x(u) \coloneqq \left( f(v) \right)_{v \sim u}$  and $w(u) \coloneqq \left( w(u,v) \right)_{v \sim u}$.
A function $f \in \Knoten_g(V)$ is called an \emph{$\infty$-harmonic extension} of $g$ if it is a fixed point of $\Midr_w$, i.e.,
\begin{equation}
f = \Midr_w f.
\end{equation}
The relation between $\infty$-harmonic extensions and  minimal Lipschitz ones is given by the following lemma.
The first part of the lemma also guarantees the existence of $\infty$-harmonic extensions.
\begin{lemma}\label{lem:lip=harm}
	Let $f\in \Knoten_g(V).$
	\begin{enumerate}
		\item \label{harm:i} If $f$ is a minimal Lipschitz extension, then it is $\infty$-harmonic.
		\item \label{harm:ii} For $m=1$,  if $f$ is an $\infty$-harmonic extension, then it is uniquely determined 
		and coincides therefore with the minimal Lipschitz extension.
	\end{enumerate}
\end{lemma}
\begin{proof}
	\eqref{harm:i} Let $f \in \Knoten_g(V)$ be the minimal Lipschitz extension.
	Assume that $f$ is not $\infty$-harmonic. 
	Then there exist $u_0 \in V \setminus U$ such that
	\begin{equation*}
	f(u_0) \neq a_0 \coloneqq \argmin_{a \in \R^m} \max_{v \sim u_0} w(u_0,v) \l| a - f(v)\r|.
	\end{equation*}
	Define $h \in \Knoten_g(V)$ as
	\begin{equation*}
	h(u) \coloneqq \left\{
	\begin{array}{ll}
	f(u)& \mathrm{if} \; u \neq u_0,\\
	a_0 & \mathrm{if} \; u = u_0.
	\end{array}
	\right.
	\end{equation*}
	Then we have $Lf(u_0) > Lh(u_0)$ and $Lf(u) = Lh(u)$ for all $u \in V \setminus U$
	with $u \not \sim u_0$.
	Assume that $Lf(u) < Lh(u)$ for some $u \neq u_0$, $u \sim u_0$.
	Since
	\begin{equation*}
	w(u,v) |f(u) - f(v)| = w(u,v) |h(u) - h(v)|,
	\end{equation*}
	for every $v \sim u$, $v \neq u_0$, we have
	$u_0 = \argmax_{v \sim u} w(u,v) |h(u) - h(v)|$.
	Thus, $Lh(u) \le Lh(u_0)$. 
	Now we get
	\begin{align}
	\max \left\{ Lf(u)\colon u \in V\setminus U, \, Lf(u) > Lh(u) \right\} 
	&\ge 
	Lf(u_0) > Lh(u_0)\\
	&\ge 
	\max \left\{ Lf(u)\colon u \in V\setminus U, \, Lh(u) > Lf(u) \right\} .
	\end{align}
	This yields the contradiction $\Llex(h) < \Llex(f)$.
	
	\eqref{harm:ii} Follows from Theorem~\ref{inf-lap=inf-harm}. For the non-weighted case, see also~\cite{PSSW2005}.
\end{proof}

The second statement of the lemma is in general not true for $m \ge 2$.
Moreover, in the vector-valued case, an $\infty$-harmonic extension $f\in \Knoten_g(V)$ is not necessarily a minimizer of $E_\infty$.
An example is given in Fig.~\ref{fig:disc_inf-harm_not_unique}, and a more sophisticated one in \cite{SS2012}.
Also, if an $\infty$-harmonic extension is a minimizer of $E_\infty$ it doesn't have to be Lipschitz minimal.
\begin{figure}[t]
	\centering
	\includegraphics[width=0.27\linewidth]{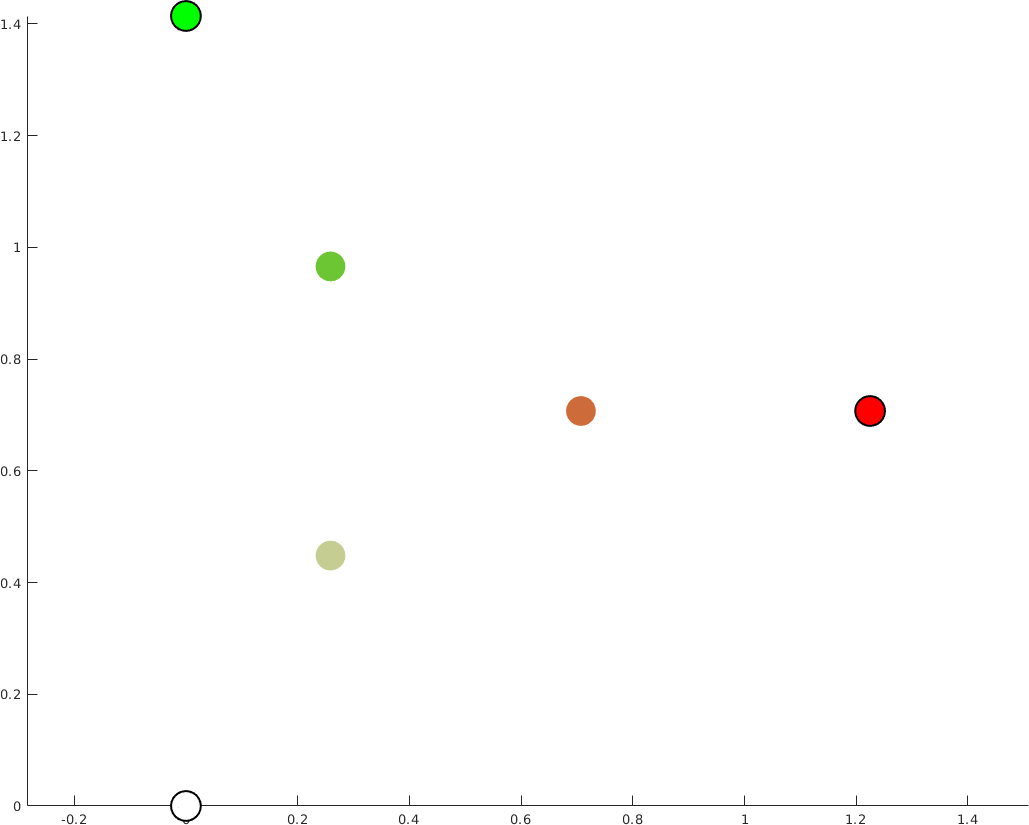}
	\hskip.5em
	\includegraphics[width=0.27\linewidth]{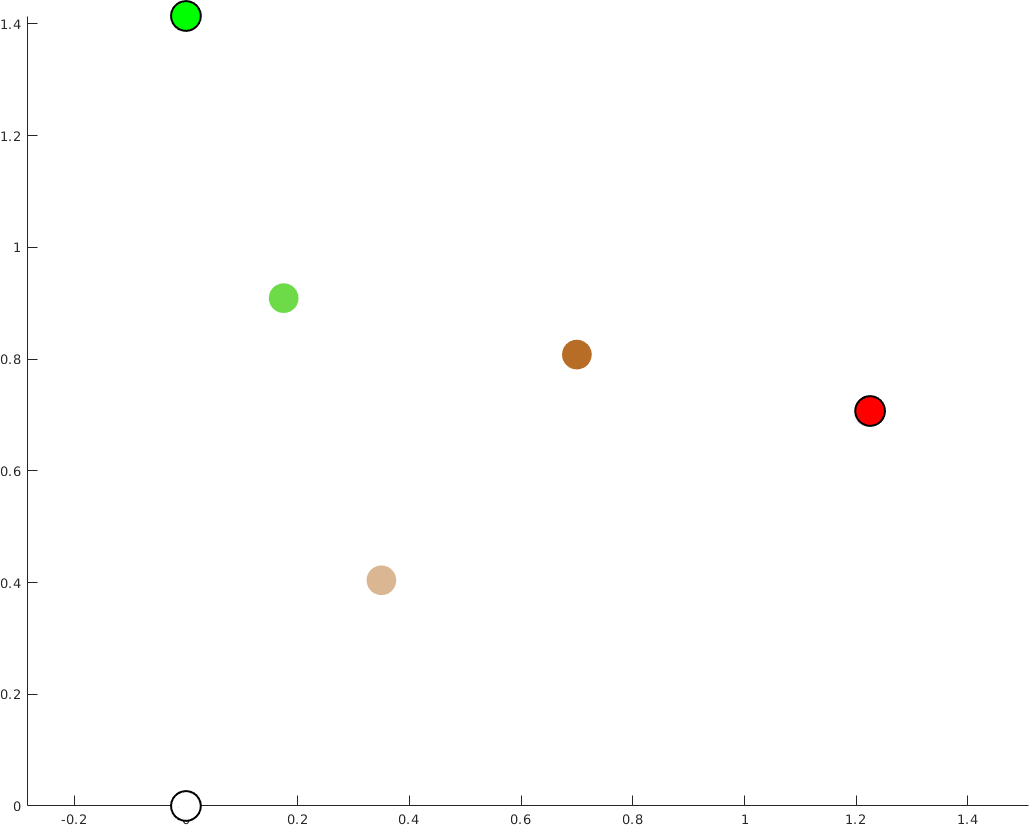}
	\hskip.5em
	\includegraphics[width=0.27\linewidth]{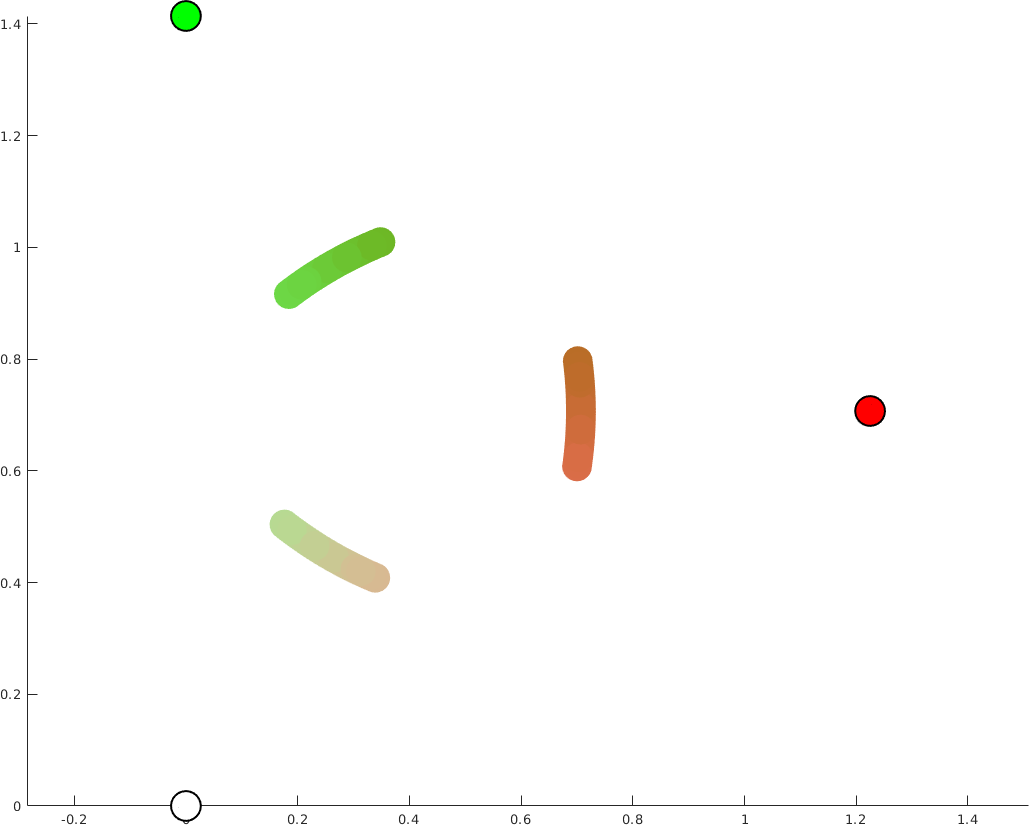}
	\caption{Discrete $\infty$-harmonic extensions of the points $(0,0,0)$ (white), 
		$(1,0,0)$ (red) and $(0,1,0)$ (green) in the RGB color cube visualized in the triangle plane.
		Left: minimal Lipschitz extension, Middle: some $\infty$-harmonic extension, Right: set of all $\infty$-harmonic extensions.}
	\label{fig:disc_inf-harm_not_unique}
\end{figure}
For the scalar-valued case, we can use the fact that $\Midr_w$ is nonexpansive to deduce an iterative algorithm 
for computing the minimal Lipschitz extension of $g$. 
To this end, we use  Krasnoselskii--Mann iterations. 
This method is in general known to converge in Hilbert spaces, but cannot be generalized to Banach spaces without additional constraints.
Fortunately, the Krasnoselskii--Mann method works for finite dimensional normed spaces~\cite[Corollaries~10,11]{BoReSh92}.
\begin{theorem}[Krasnoselskii--Mann Iteration] \label{thm:KM}
	Let $T \colon \R^N \to \R^N$ be a nonexpansive mapping
	with nonempty fixed point set, where $\R^N$ is equipped with an arbitrary norm.
	Then, for every starting point $f^{(0)} \in \R^N$, the sequence of iterates $\{f^{(r)} \}_{r \in \mathbb N}$ generated by
	\begin{equation*}
	f^{(r+1)} \coloneqq \left( (1-\tau_r) I + \tau_r T \right) f^{(r)}, \qquad \tau_r \in (0,1),
	\end{equation*}
	converges to a fixed point of $T$ provided that $\sum_{r=1}^\infty \tau_r = \infty$ and $\limsup_{r \to \infty} \tau_r < 1$. 
\end{theorem}
We apply Theorem \ref{thm:KM} with $T \coloneqq \Midr_w$. 
For given $g\colon U \to \R^m$, $f^{(0)} \in \Knoten_g(V)$ and $\tau_r = \tau \in (0,1)$, 
we obtain the following iteration scheme:
\begin{align}\label{mid_iter}
f^{(r+1)} &\coloneqq 
f^{(r)} + \tau \left( \Midr_w f^{(r)} - f^{(r)}  \right).
\end{align}
In the scalar-valued case, we have an immediate convergence result.
\begin{corollary}\label{cor:conv_mid}
	Let $g\colon U \to \R$ be given.
	Then, for every $f^{(0)} \in \Knoten_g(V)$ and $\tau \in (0,1)$, 
	the sequence of iterates $\{f^{(r)} \}_{r \in \mathbb N}$ 
	generated by \eqref{mid_iter}
	converges to the minimal Lipschitz extension of $g$.
\end{corollary}
\begin{proof}
	By Lemma \ref{lem:lip}, we have for $f_1,f_2 \in \Knoten_g(V)$ that
	\begin{align}
	\| \Midr_w f_1 - \Midr_w f_2\|_\infty 
	&= 
	\max_{u \in V \setminus U} |\Midr_w f_1(u) - \Midr_w f_2(u)|\\
	&= 
	\max_{u \in V\setminus U} |\midr_{w(u)} x_1(u) - \midr_{w(u)} x_2(u)|\\
	&\le
	\max_{u \in V\setminus U} \|x_1(u)-x_2(u)\|_\infty 
	=
	\|f_1 - f_2\|_\infty.
	\end{align}
	Hence $\Midr_w \colon \R^N \to \R^N$ is nonexpansive with respect to the $\infty$-norm
	and  by Theorem \ref{thm:KM} the sequence converges to a fixed point of $\Midr_w$.
	By Lemma \ref{lem:lip=harm} there is exactly one  fixed point which is the minimal Lipschitz extension of $g$.
\end{proof}
For the vector-valued case, we apply the iteration scheme \eqref{mid_iter} as a heuristic method and obtain promising results 
in Section \ref{sec:numerics}. However, we did not succeed in proving a convergence theorem. On the positive side, we have the following result.
\begin{corollary}\label{cor:conv_mid_vector}
	Let $g\colon U \to \R^m$, $m \ge 1$ be given.
	If for $f^{(0)} \in \Knoten_g(V)$ the sequence of iterates $\{f^{(r)} \}_{r \in \mathbb N}$ generated by \eqref{mid_iter} is asympotically regular, 
	that is, it satisfies $\lim_{r \to \infty} \Vert f^{(r+1)} - f^{(r)} \Vert_{2,\infty} = 0$, 
	then every cluster point is an $\infty$-harmonic extension of $g$.
\end{corollary}
\begin{proof}
	Assume that we have a convergent subsequence $\{f^{(r_j)} \}_j$ with limit $\hat f$. 
	Since the iterates $\{f^{(r_j+1)} \}_j$ are bounded, we can choose a subsequence of $\{f^{(r_j+1)} \}_j$ 
	(again denoted with the same indices) such that also $\{f^{(r_j +1)} \}_j$ converges with limit $\tilde f$. 
	Since  $\lim_{r \to \infty} \Vert f^{(r_j+1)} - f^{(r_j)} \Vert_{2,\infty} = 0$, 
	it follows $\tilde f = \hat f$.
	Let $u\in V\setminus U.$
	Using the continuity of $\midr_{w(u)}$, we obtain
	\begin{align}
	\hat f(u) 
	&= \lim_{j \to \infty} f^{(r_j+1)} (u) = \lim_{j \to \infty} (1-\tau)f^{(r_j)}(u) + \tau \midr_{w(u)} f^{(r_j)}(u)\\
	&= (1-\tau)\hat f(u) + \tau \midr_{w(u)} \hat f(u).
	\end{align}
	Rearranging the terms yields the assertion.
\end{proof}
\section{$\infty$-Laplacians on Scalar-Valued Functions} \label{sec:inf_lap}
The minimizers of $E_p$ in \eqref{E_p} are determined by the zero of the \emph{anisotropic $p$-Laplacian operator}
\begin{equation} \label{grad}
\Delta_{w,p} f(u) \coloneqq  \sum_{v \in V} w(u,v)^p |f(u) - f(v)|^{p-2} \big( f(u) - f(v) \big) =0,
\end{equation}
for $u \in V \setminus U$.
To the best of our knowledge there exists no satisfactory definition of an $\infty$-Laplacian for vector-valued functions. 
There was an attempt in this direction in~\cite{BT2017} which is unfortunately not well-defined.

For scalar valued-functions, the $\infty$-Laplacian $\Delta_{w,\infty}\coloneqq \Knoten_g(V) \to \Knoten_g(V)$ is given by
\begin{align}
\Delta_{w,\infty} f(u) 
&\coloneqq
\frac{1}{2} \Big(  \max_{v \sim u} \ w(u,v) \left( f(v) - f(u) \right) 
+ \min_{v \sim u} w(u,v)\left( f(v) - f(u) \right) 
\Big) \\
&=
\frac{1}{2} \Big( \max_{v \sim u}  w(u,v) \left( f(v) - f(u) \right)  
- \max_{v \sim u} w(u,v)\left( f(u) - f(v) \right) \Big) \label{lap_inf}
\end{align}
for $u \in V \setminus U$. This definition is due to \cite{PSSW2005} and for weighted graphs due to \cite{EDLL2014,ETT2015}.
Note that $w(u,u) > 0$ implies  $\max_{v \sim u} \ w(u,v) \left( f(v) - f(u) \right) \ge 0$. 
For non-weighted graphs, we have 
\begin{equation} \label{eq_inf}
\Delta_{\infty} f(u) = \frac{1}{2} \max_{v \sim u} \left( f(v) - f(u) \right)  + \frac{1}{2} \min_{v \sim u} \left( f(v) - f(u) \right)
= \midr x(u) - f(u).
\end{equation}
Hence, on $\Knoten_g(V)$, the zero of $\Delta_{\infty}$ coincide with the fixed point of $\Midr$,
that is, with the unique $\infty$-harmonic extension of $g$. This observation extends into the weighted case as well. Indeed, it was showed already in \cite{EDLL2014} that, 
given $g\colon U\to\R,$ there exists a unique extension $f\in \Knoten_g(V)$ with $\Delta_{w,\infty}f(u)=0$ for every $u\in V\setminus U.$ We now obtain the desired statement.
\begin{theorem}\label{inf-lap=inf-harm}
	Let $m=1$ and $g\colon U\to\R.$ Then $f\in \Knoten_g(V)$ satisfies $\Delta_{w,\infty} f(u) = 0$ for every $u \in V \setminus U$ if and only if 
	it is an $\infty$-harmonic extension of~$g$ and it is the case if and only if it is the minimal Lipschitz extension of $g$.
\end{theorem}
\begin{proof}
	On account of the above discussion it remains to prove the following. 
	If $f \in  \Knoten_g(V)$ fulfills $\Delta_{w,\infty} f(u) = 0$ for all $u \in V \setminus U$, then $f$ is $\infty$-harmonic. 
	Since  
	\begin{align}
	\max_{v \sim u} \ w(u,v) \left( f(v) - f(u) \right) = \max_{v \sim u} w(u,v)\left( f(u) - f(v) \right)
	= \max_{v \sim u} \ w(u,v) \left| f(v) - f(u) \right|,
	\end{align}
	we obtain for $a \ge 0$ that
	\begin{align}
	\max_{v \sim u} w(u,v) \left|f(v) - \left( f(u) + a \right) \right|
	&\ge
	\max_{v \sim u} w(u,v)  \max \l\{ 0, f(u) -  f(v) + a\r\} \\
	&\ge 
	\max_{v \sim u} w(u,v) \left(f(u) -  f(v)\right) \\ &=
	\max_{v \sim u} w(u,v) \left| f(v) -  f(u) \right|.
	\end{align}
	Similarly, for $a \le 0,$ we obtain the inequality
	\begin{equation*}
	\max_{v \sim u} w(u,v) \left|f(v) - \left( f(u) + a \right) \right| \ge \max_{v \sim u} w(u,v) \left| f(v) -  f(u) \right|.
	\end{equation*}
	This finishes the proof.
\end{proof}

For given $g\colon U \to \R$ and $f^{(0)} \in \Knoten_g(V)$ 
we consider the iteration scheme
\begin{equation}\label{inf_lap_iter}
f^{(r+1)} \coloneqq 
f^{(r)}  + \tau \Delta_{w,\infty} f^{(r)} .
\end{equation}
This scheme was proposed in \cite{EDLL2014,ETT2015} without a convergence proof.
The authors only proved that in case of convergence the sequence converges to a zero of the $\infty$-Laplacian.

In the non-weighted case, this can be rewritten by \eqref{eq_inf} as 
\begin{equation*}
f^{(r+1)}(u) =  (1-\tau) f^{(r)} (u) + \tau  \midr x^{(r)}(u),\qquad u \in V \setminus U.
\end{equation*}
By \eqref{mid_iter} and Corollary \ref{cor:conv_mid} we see that the sequence of iterates \eqref{inf_lap_iter} 
converges to the minimal Lipschitz extension of $g$ for $\tau \in (0,1)$.

For the weighted setting, this is also true by the following corollary.
\begin{corollary}\label{cor:conv_inf_lap}
	Let $g\colon U \to \R$ be given.
	Then, for every $f^{(0)} \in \Knoten_g(V)$ and $\tau \in (0,1)$, 
	the sequence of iterates $\{f^{(r)} \}_{r \in \N}$ 
	generated by \eqref{inf_lap_iter}
	converges to the minimal Lipschitz extension of $g$.
\end{corollary}
\begin{proof}
	Consider the operator $\Phi\colon\Knoten_g(V) \to \Knoten_g(V)$ defined by 
	\begin{equation}
	\Phi f  \coloneqq
	f + \Delta_{w,\infty} f.
	\end{equation}
	By virtue of Theorem \ref{inf-lap=inf-harm}, the mapping $\Phi$ has a unique fixed point determined by the zero of the $\infty$-Laplacian.
	We show that $\Phi$ is nonexpansive. Then it follows immediately from Theorem \ref{thm:KM} that the sequence $\{f^{(r)} \}_{r \in \mathbb N}$ converges to this fixed point.
	For $u\in V \setminus U$ we rewrite 
	\begin{equation}
	\Delta_{w,\infty} f(u) =
	\frac{1}{2} \max_{v \sim u} w(u,v) \left( f(v)-f(u) \right)
	- 
	\frac{1}{2}
	\max_{v \sim u} w(u,v) \left( f(u)-f(v) \right).
	\end{equation}
	For $i=1,2,$ let $f_i\in \Knoten_g(V)$. Now, define $y_i \coloneqq \argmax_{v \sim u} w(u,v) \left( f_i(v)-f_i(u) \right)$ together with
	$z_i \coloneqq \argmax_{v \sim u} w(u,v) \left( f_i(u)-f_i(v) \right)$.
	Then, for $u\in V \setminus U$, we obtain 
	\begin{align}
	&\Phi f_1 (u) - \Phi f_2  (u)
	=
	f_1(u) - f_2(u) + \Delta_{w,\infty} f_1(u) - \Delta_{w,\infty} f_2(u)\\
	&=
	f_1(u) - f_2(u)  
	+ \frac{w(u,y_1)}{2} \left( f_1(y_1) -f_1(u) \right)
	- \frac{w(u,z_1)}{2}  \left( f_1(u) - f_1(z_1) \right)\\
	&\quad - \frac{w(u,y_2)}{2}  \left( f_2(y_2) -f_2(u) \right)
	+ \frac{w(u,z_2)}{2}  \left( f_2(u) - f_2(z_2) \right)\\
	&\leq
	f_1(u) - f_2(u) + 
	\frac{w(u,y_1)}{2} \left( f_1(y_1) - f_1(u) - f_2(y_1) + f_2(u) \right)\\
	&\quad - 
	\frac{w(u,z_2)}{2} \left(f_1(u) - f_1(z_2) - f_2(u) + f_2(z_2) \right)\\
	&= 
	\left( 1-\frac{w(u,y_1)+w(u,z_2)}{2} \right) \left( f_1(u) - f_2(u) \right) \\
	&\quad + 
	\frac{w(u,y_1)}{2} \left( f_1(y_1) - f_2(y_1) \right)
	+\frac{w(u,z_2)}{2} \left( f_1(z_2) - f_2(z_2) \right)
	\end{align}
	and with $\|f_1-f_2\|_\infty = \max_{u \in V} |f_1(u) -f_2(u)|$ further
	$
	\Phi f_1 (u) - \Phi f_2  (u)
	\leq 
	\| f_1-f_2\|_{\infty}.
	$
	Analogously, we get 
	$
	\Phi f_1(u) -\Phi f_2(u) \geq - \| f_1-f_2\|_{\infty}
	$
	and thus
	$
	\|\Phi f_1 - \Phi f_2\|_{\infty} \leq \| f_1-f_2\|_{\infty}.
	$
\end{proof}

\section{Numerical Aspects of Approximating Algorithms} \label{sec:algs}
In this section, we comment on numerical computations of the minimizer of $E_p$ for large $p \in \mathbb N$.
By Theorem \ref{thm:ConvEp} such a minimizer can be considered as an approximation of the minimal Lipschitz extension of $g$.
As in the scalar-valued case, we refer to it as Polya's method.
Using a preconditioned Newton method, computations are performed up to $p=2400$ in Section~\ref{sec:numerics}. 
Let us mention that we could also compute the minimizer of $E_{\infty,p}$, e.g.,~by the ADMM method described in \cite{Hertrich2018}.
However, this method is more time consuming and works so far only for moderate sizes of $p \le 25$. 
We also apply the iterated midrange filter from Section \ref{sec:midrange+harm}, but in a Gauss-Seidel like fashion.
\subsection{Preconditioned Newton method}\label{sec:Newton}
Since $E_p$ is a smooth functional, the minimizers can be computed using Newton's method.
This concept was also pursued for the continuous setting in \cite{KatPry16}.
\begin{theorem} \cite[Theorem 3.7]{NocWri06}
	\\
	Let $F \in C^2(\R^N,\R)$ have a unique minimizer $x^*$ with Lipschitz continuous Hessian 
	$\nabla^2 F$ in a neighborhood of $x^*$.
	Further, assume that $\nabla^2 F$ is positive definite at $x^*$.
	If the initial guess $x^{(0)}$ is sufficiently close to $x^*$, the iteration
	\begin{align}
	\nabla^2 F(x^{(r)}) \, h^{(r)} &= \nabla F(x^{(r)}), \label{system}\\
	x^{(r+1)} &= x^{(r)} - h^{(r)}  \label{update}
	\end{align}
	converges quadratically to $x^*$.
\end{theorem}
A global convergence result under some stronger assumptions can be found in 
\cite[Theorem~6.3]{NocWri06} or \cite[Section~9.5.3]{BoyVan04}.
The idea is to change the Newton scheme to
\begin{equation}\label{newton_mod}
x^{(r+1)} = x^{(r)} - \alpha_r \nabla^2 F(x^{(r)})^{-1} \nabla f(x^{(r)}),
\end{equation}
where $\alpha_r$ is computed with an Armijo backtracking line search which always tries the step size $\alpha_r =1$ first.
If $\mathcal{K} = \{x \in \mathbb R^N \colon F(x) \leq F(x_0)\}$ is compact and there exists a constant $C$ such that the condition number 
\begin{equation}\label{got_it}
\text{cond}\bigl(\nabla^2f(x^{(r)})\bigr) \leq C
\end{equation}
for all $r$, then the scheme is globally convergent to $x^*$.
For $r$ large enough, the step size is always chosen as $\alpha_r =1$ so that the convergence becomes quadratic.
In our numerical examples, we have observed \eqref{got_it}.

In order to apply Newton's method to $E_p$, we need to compute its Hessian.
Differentiating the gradient of $E_p$ in \eqref{grad}, we observe that the Hessian is block structured with diagonal blocks
$\left( \frac{\partial^2 E_p} {\partial f_i(u) \partial f_j(u)} \right)_{i,j=1}^m$
and non-diagonal blocks 
$\left( \frac{\partial^2 E_p} {\partial f_i(u) \partial f_j(v)} \right)_{i,j=1}^m$ for $u \sim v$, $u \not = v$, where 
{\small
	\begin{align}
	&\frac{\partial^2 E_p} {\partial f_i(u)^2} =  \sum_{v \in V} w(u,v)^p \Bigl( |f(u) - f(v)|^{p-2} + (p-2) |f(u) - f(v)|^{p-4} \big( f_i(u) - f_i(v) \big)^2\Bigr),\\
	&\frac{\partial^2 E_p} {\partial f_i(u) \partial f_j(u)} =  \sum_{v \in V} w(u,v)^p  (p-2)|f(u) - f(v)|^{p-4} \big( f_i(u) - f_i(v) \big)\big( f_j(u) - f_j(v) \big),\\
	&\frac{\partial^2 E_p} {\partial f_i(u) \partial f_i(v)} =  - w(u,v)^p (p-2) \Bigl(|f(u) - f(v)|^{p-4} \big( f_i(u) - f_i(v) \big)^2 + |f(u) - f(v)|^{p-2}\Bigr),\\
	&\frac{\partial^2 E_p} {\partial f_i(u) \partial f_j(v)} =  - w(u,v)^p  (p-2)|f(u) - f(v)|^{p-4} \big( f_i(u) - f_i(v) \big)\big( f_j(u) - f_j(v) \big).
	\end{align}
}
For large $p$ the factor $w(u,v)^p |f(u) - f(v)|^{p-4}$ can get very large, resp., very small, possibly causing 
a bad condition number of the Hessian.
Therefore, the choice of a suitable preconditioner is crucial for solving the linear system of equations in \eqref{system}.
One possible preconditioner choice is the diagonal matrix $D$ with entries
\begin{equation}
D_{f_i(u),f_i(u)} = \frac{c}{\max_{v \in V}\bigl(w(u,v) |f(u) - f(v)|\bigr)^{p-2}}, \quad c \in (0,1]
\end{equation}
in the diagonal block related to $u \in V \setminus U$.
This matrix ensures that for every $f(u)$ at least one edge has numerically reasonable values.
Then, we solve 
\begin{equation}\label{eq:HessianDirection}
D \; \nabla^2 E_p \; h = - D \; \nabla E_p
\end{equation}
instead of \eqref{system}.
In our implementation, the minimizers of $E_p$ are computed for increasing $p$ with the previous minimizer as an initialization.
We always choose $p=2$ as a starting point, since this problem reduces to solving a linear system.
It was not necessary to update the preconditioner after every Newton step and the one from the first step is used for all iterations. 
Since \eqref{eq:HessianDirection} is non symmetric, we choose \emph{bicgstab} with Gauss-Seidel or Jacobi preconditioning as a linear solver. 
For better performance the problem can be solved on a GPU.
This naturally rises the question if a symmetric preconditioning of \eqref{system} exists which circumvents the numerical instabilities.

\subsection{Iterated Midrange Filter}
Based on the ideas in Section~\ref{sec:midrange+harm} a second approach to ,,compute'' the minimal Lipschitz extension of $g$ is to apply the iterated midrange filter with an appropriate starting point.
As starting point $f^{(0)}$ we use again the minimizer of $E_2$, i.e.~the zero of the $2$-Laplacian.
In contrast to Section~\ref{sec:midrange+harm}, we apply the midrange filter in a cyclic or Gauss-Seidel like fashion.
If the vertices in $V \setminus U$ are numbered from $0$ to $N$ with 
$c(k) \coloneqq \text{mod}(r,N)$, one cycle of the algorithm reads for $\tau \in (0,1]$ as follows
\begin{equation}\label{eq:CyclicUpdate}
f^{(k+1)}(v) \coloneqq \left\{
\begin{array}{ll}f^{(k)}(v) + \tau \left( \midr_{w(v)} x^{(k)}(v) - f^{(k)}(v) \right) & \text{ if } v=u_{c(k)},\\
f^{(k)}(v) & \text{ otherwise.}
\end{array}\right.
\end{equation}
Clearly, the vertices can be also visited in a random cyclic order.
For the case $m=1$, a convergence result follows similar as in Corollary~\ref{cor:conv_mid} with minor modifications due to the cyclic update.
Unfortunately, the proof does not generalize to the vector-valued case, because the midrange filter is not non-expansive for $m\geq 2$.
However, if the sequence converges, we can adapt Corollary~\ref{cor:conv_mid_vector} to the cyclic setting and observe that the scheme provides an $\infty$-harmonic extension of $g$.
Additionally, we have the following ,,descent'' property. 
\begin{proposition}
	The update $f^{(k+1)}$ in \eqref{eq:CyclicUpdate} is $\Llex$ smaller than $f^{(k)}$ if $f^{(k+1)} \neq f^{(k)}$.
\end{proposition}
\begin{proof}
	Assume $f^{(k+1)}\neq f^{(k)}$ and let $u$ denote the updated vertex in step $k$.
	Then, it holds
	\begin{align}
	Lf^{(k+1)}(u) 
	&= \max_{v \sim u \atop v \not = u} w(u,v)| f^{(k+1)}(u) - f^{(k)}(v) | 
	\leq
	\max_{v \sim u} w(u,v)| f^{(k+1)}(u) - f^{(k)}(v) |\\
	&\leq 
	(1-\tau) Lf^{(k)}(u) + \tau \max_{v \sim u} w(u,v) |  \midr_{w(u)} x^{(k)}(u) - f^{(k)}(v) | \\
	&=
	(1-\tau) Lf^{(k)}(u) + \tau \max_{v \sim u} w(u,v) \big| \argmin_a (\max_{\tilde v \sim u} w(u,\tilde v) |a - f^{(k)}(\tilde v)|)  - f^{(k)}(v) \big|\\
	&< Lf^{(k)}(u),
	\end{align}
	where the last inequality follows by $f^{(k+1)} \neq f^{(k)}$.
	For $v \sim u$, $v \not = u$, we obtain by
	\begin{align}
	Lf^{(k+1)}(v) 
	&= \max\big\{ \max_{\tilde v \sim v \atop \tilde v \not = u} w(v,\tilde v)| f^{(k)}(v) - f^{(k)}(\tilde v) | ,w(v,u)| f^{(k)}(v) - f^{(k+1)}(u) | \big\}
	\end{align}
	that 
	$L f^{(k+1)}(v) \le \max \{Lf^{(k)}(v), L f^{(k+1)}(u)\}$. 
	Hence, for all $v\in V\setminus U$ with $Lf^{(k)}(v) \geq Lf^{(k)}(u)$ 
	it holds 
	$L f^{(k+1)}(v) \leq Lf^{(k)}(v)$ and if $Lf^{(k)}(v) < Lf^{(k)}(u)$ 
	it holds $L f^{(k+1)}(v) < Lf^{(k)}(u)$.
	Since $Lf^{(k+1)}(u)<Lf^{(k)}(u)$, this implies that $f^{(k+1)}$ is $\Llex$ smaller than $f^{(k)}$.
\end{proof}

Obviously, the previous proposition implies that the sequence of iterates stays bounded.

\section{Numerical Examples} \label{sec:numerics}
%
In this section, proof-of-the-concept examples for the performance of the proposed approaches are provided.
To outline the differences between the approximation of the minimal Lipschitz extension by Polya's method, iterated midrange filters and componentwise minimal Lipschitz extensions, we start with an intuitive example, where the function is defined on a simple graph and maps into $\mathbb R^2$, i.e., $m=2$.
Then, we consider the inpainting of RGB images, i.e.~of vector-valued functions with $m=3$.
An original image $f \in [0,1]^{M,N}$ has missing pixels, 
so that the function/image values are only known on a subset $U \subset \{1,\dots,M\} \times \{1,\dots,N\}$.
The aim is to reconstruct the complete image.
In the first few examples, we assume that the function lives on a 4-neighborhood graph of the image grid.
Of course, minimal Lipschitz extensions for inpainting tasks make more sense if the function values on neighboring vertices of the graph are similar.
This can be achieved by applying nonlocal patch-based techniques.
Starting with the pioneering work of \cite{BCM05} such methods were successfully refined and applied for different tasks in image processing.
Concerning Polya's method, we apply the method described in Section~\ref{sec:Newton} with $p=2,5,10,15,20$ and then further increase $p$ in steps of 10 until the desired value of $p$ is reached.
In particular, we initialize with the with the result obtained from the previous $p$ value.
As initialization for computing the iterated midrange filter the $2$-Laplacian is used.
All algorithms are implemented in Matlab.

\paragraph{2D function.}
The first example in Fig.~\ref{fig:Grid} shows a simple graph $G=(V,E)$ with equal weights, where the blue lines indicate the edges $E$ between vertices $V$.
The blue circles illustrate the fixed values of the function $g\colon U \rightarrow \mathbb R^2$ as spatial positions in $\mathbb R^2$.
The red crosses are the inpainted values of $f \in \Knoten_g(V)$.
For this simple example, the minimal Lipschitz extension can be computed analytically and coincides with the result of the iterated midrange filter \eqref{mid_iter} with $\tau =0.95$.
However, no convergence result is available in the vector-valued case so far and the result is only experimental.
Note that for computing the individual midrange filters we applied the simple procedure described in \cite[Section~4]{Phan2016}.
For larger size problems more sophisticated methods as described in Subsection~\ref{subsec:midr} should be used.
For Polya's algorithm we observe that if $p$  increases, the solution of the $p$-Laplacian indeed converges to the minimal Lipschitz extension.
Finally, we computed the  result of the componentwise minimal Lipschitz extension by \eqref{inf_lap_iter}, see Corollary~\ref{cor:conv_inf_lap}.
The function differs completely from the minimal Lipschitz extension of the vector-valued function.
This is not really surprising: Consider for example a graph consisting of four vertices, where one vertex is connected to the remaining three vertices by edges with equal weights. 
Let the function values on those three vertices be $(0,0), (1,0), (1/2,\sqrt{3}/2)$ (equilateral triangle), the minimal Lipschitz extension is $(1/2,\sqrt{3}/6)$ (the circumcenter of the triangle),
while the componentwise minimal Lipschitz extensions yields the point $(1/2,\sqrt{3}/4)$.
\begin{figure}[t]
	\centering
	\parbox{\figrasterwd}{
		\centering
		\parbox{.3\figrasterwd}{%
			\subfloat[][2-Laplacian]{\includegraphics[width=\hsize]{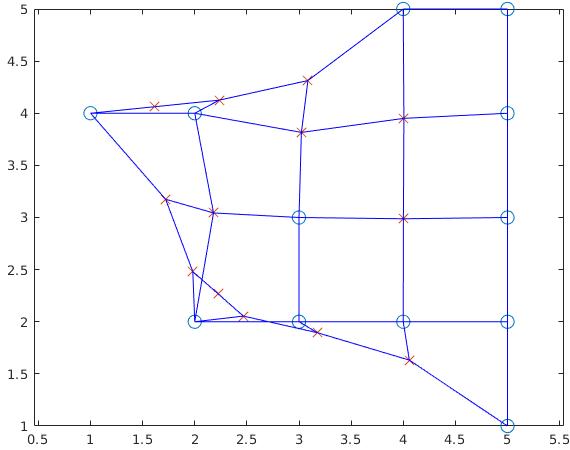}}
			
			\subfloat[][2400-Laplacian]{\includegraphics[width=\hsize]{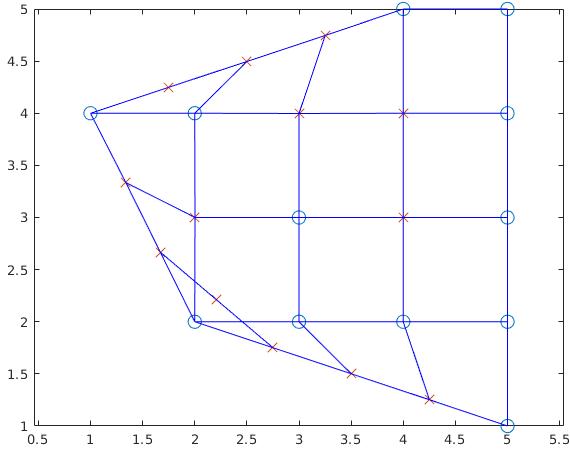}}
		}
		\hskip .5em
		\parbox{.3\figrasterwd}{%
			\subfloat[][50-Laplacian]{\includegraphics[width=\hsize]{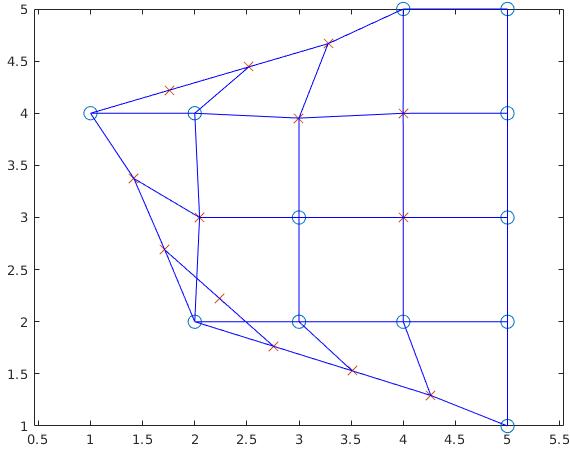}}
			
			\subfloat[][Iterated midrange filter]{\includegraphics[width=\hsize]{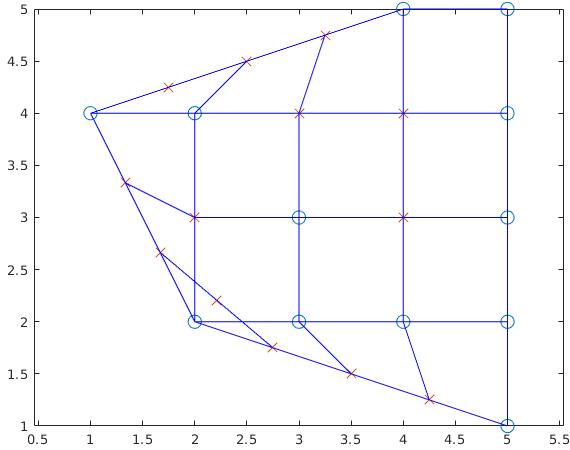}}
		}
		\hskip .5em
		\parbox{.3\figrasterwd}{%
			\subfloat[][200-Laplacian]{\includegraphics[width=\hsize]{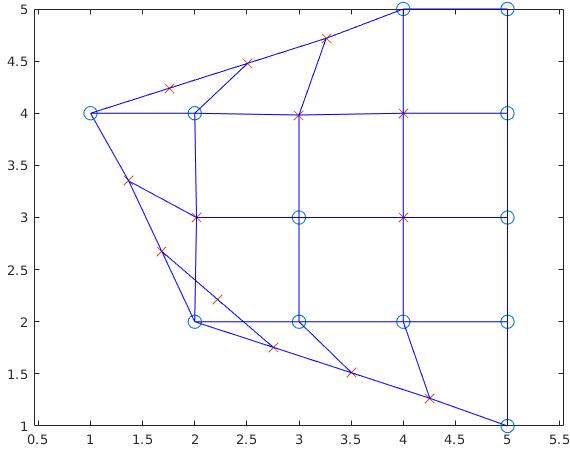}}
			
			\subfloat[][Comp.~minimal Lipschitz]{\includegraphics[width=\hsize]{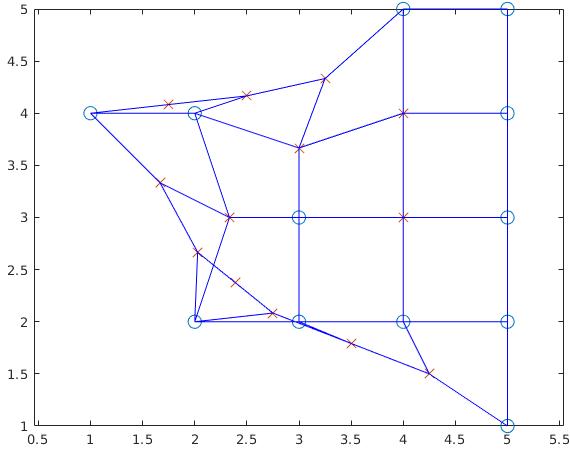}}
		}
	}		
	\caption{Comparison of different extensions of a 2D function with given blue circle values.
		The solution of $p$-Laplacians for large $p$ as well as the iterated midrange filter approximate the minimal Lipschitz extension well. The componentwise, single-valued minimal Lipschitz extensions differ completely from those extensions.}
	\label{fig:Grid}
\end{figure}

\paragraph{RGB inpainting on a local neighborhood graph.}
Next, we consider a simple RGB image with a missing square in the center, where the four color values form a square in the RGB cube.
We use an equally weighted 4-neighborhood graph on the image grid.
Fig.~\ref{fig:RGBSquare} shows the inpainting results with the same methods as in the previous example.
Again the solution of the $p$-Laplacian for large $p$ and the iterated midrange filter lead to nearly the same images which differ from the comp.~minimal Lipschitz solution.
At first glance the results might look a bit unexpected since the 2-Laplacian appears to be smoother than the others.
In order to check that the values of the Lipschitz constants in $\Llex$ get indeed lexicographically smaller, we added a table with the 10 largest Lipschitz constants for increasing $p$ in Fig.~\ref{fig:Lipschitz}.
This table also contains the values for the comp.~minimal Lipschitz extension in the last column, which are already worse than the values of the $40$-Laplacian.
\begin{figure}
	\centering
	\parbox{\figrasterwd}{
		\centering
		\parbox{.3\figrasterwd}{%
			\subfloat[][2-Laplacian]{\includegraphics[width=\hsize]{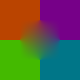}}
			
			\subfloat[][1700-Laplacian]{\includegraphics[width=\hsize]{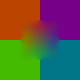}}
		}
		\hskip .5em
		\parbox{.3\figrasterwd}{%
			\subfloat[][10-Laplacian]{\includegraphics[width=\hsize]{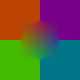}}
			
			\subfloat[][Iterated midrange filter]{\includegraphics[width=\hsize]{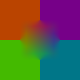}}
		}
		\hskip .5em
		\parbox{.3\figrasterwd}{%
			\subfloat[][50-Laplacian]{\includegraphics[width=\hsize]{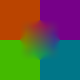}}
			
			\subfloat[][Comp.~minimal Lipschitz]{\includegraphics[width=\hsize]{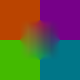}}
		}
	}	
	\caption{Comparison of different extensions of an RGB image considered as a 3D function on a 4-neigborhood grid graph.}
	\label{fig:RGBSquare}
\end{figure}
\begin{figure}
	\begin{center}
		\includegraphics[width=.98\textwidth]{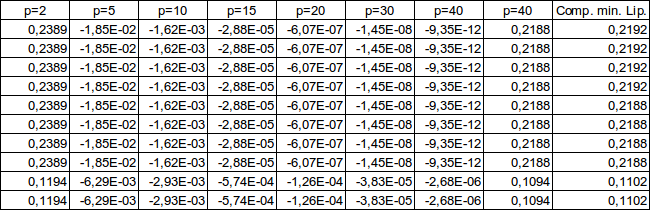}	
		\caption{Values of the 10 largest Lipschitz constants   in $\Llex$ for $p=2,40$, 
			and their changes to the previous ones for $p=5,10, \ldots,40$.}
		\label{fig:Lipschitz}		
	\end{center}
\end{figure}

\paragraph{Nonlocal image inpainting (random mask).}
In our next two examples, 90\% of the image pixels are missing, where the pixels are chosen randomly.
We present inpainting results for a more sophisticated graph choice
based on nonlocal patch similarities.
To this end, we built a graph $G$ connecting the image grid points in a semi nonlocal way.
Given some patch radius $r$, the local patch $p_{ij}$ around some pixel $(i,j)$ is defined as the quadratic part of $I$ 
with size $2r+1 \times 2r+1$ which is centered at the pixel $(i,j)$.
Then, the distance between two grid points $(i,j)$ and $(j,k)$ is defined as
\[
d\bigl((i,j),(k,l)\bigr)^2 \coloneqq \|p_{ij} - p_{kl}\|_F^2 + c \cdot \left(\frac{i-k}{m}\right)^2 + c \cdot \left(\frac{j-l}{n}\right)^2,
\]
where $\|\cdot \|_F$ denotes the Frobenius norm.
In order to reduce the computational effort, the distances are only computed in a neighborhood of radius $R$ around $(i,j)$.
The edge weights are defined as
\[w\bigl((i,j),(k,l)\bigr) \coloneqq \exp \left( - d\bigl((i,j),(k,l)\bigr)^2 \bigl( \frac{1}{\sigma(i,j)^2} + \frac{1}{\sigma(k,l)^2} \bigr)\right),\]
where $\sigma(i,j)$ is the distance of $(i,j)$ to its 20th nearest neighbor.
Note that the sum of $\sigma(i,j)$ and $\sigma(k,l)$ ensures symmetry of the adjacency matrix of $G$.
The weights are truncated to the $K$ largest ones in order to make the adjacency matrix sparse.
In our numerical experiments it turned out to be beneficial to add a small local 4-neighbor grid graph to $G$ in the first few iterations.

At the beginning, the missing pixels are assigned random Gaussian numbers with mean and covariance of the known part of the image as proposed, e.g., in~\cite{ShOsZh16}.
Then, the nonlocal graph is generated and the $200$-Laplacian extension is computed as an approximation of the minimal Lipschitz extension.
This step, including the grid generation, is repeated 15 times and the corresponding results are shown for two different images 
in Figs.~\ref{fig:peppers} and \ref{fig:simpsons}, where
we used $c=9$, $r=5$, $R=30$ and $K=40$.
We observe that higher order Laplacians perform much better than the 2-Laplacian.
\begin{figure}[t]
	\centering
	\parbox{\figrasterwd}{
		\centering
		\parbox{.38\figrasterwd}{%
			\subfloat[][Original image]{\includegraphics[width=\hsize]{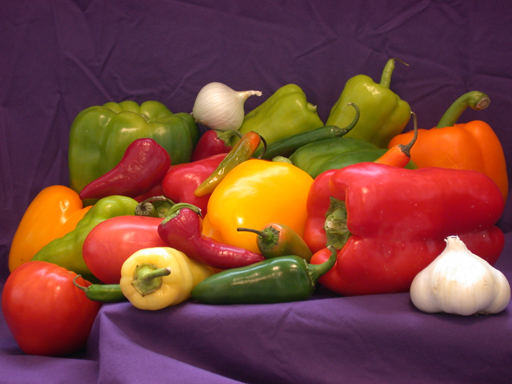}}
			
			\subfloat[][2-Laplacian (PSNR 26.14)]{\includegraphics[width=\hsize]{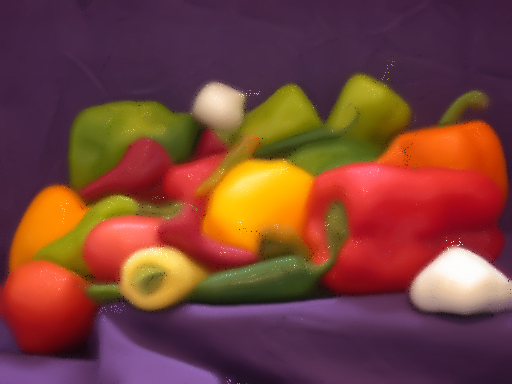}}
		}
		\hskip .5em
		\parbox{.38\figrasterwd}{%
			\subfloat[][10\% random samples]{\includegraphics[width=\hsize]{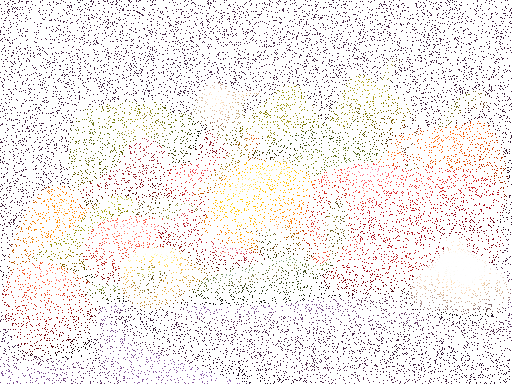}}
			
			\subfloat[][200-Laplacian (PSNR 28.87)]{\includegraphics[width=\hsize]{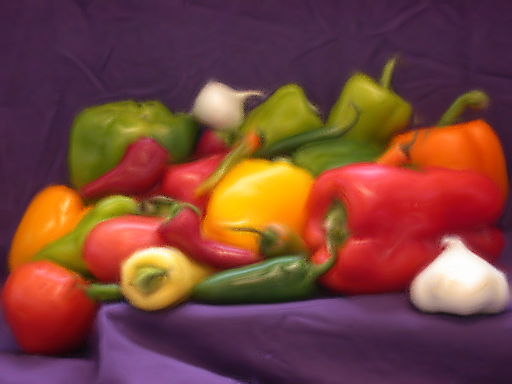}}
		}
	}	
	\caption{Nonlocal inpainting using different extensions.
		The $200$-Laplacians produces  considerably better results than the 2-Laplacians.}
	\label{fig:peppers}
\end{figure}
\begin{figure}[t]
	\centering
	\parbox{\figrasterwd}{
		\centering
		\parbox{.38\figrasterwd}{%
			\subfloat[][Original image]{\includegraphics[width=\hsize]{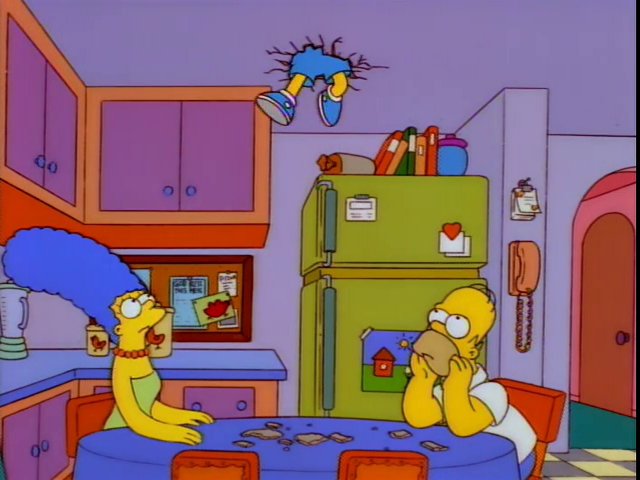}}
			
			\subfloat[][2-Laplacian (PSNR 26.14)]{\includegraphics[width=\hsize]{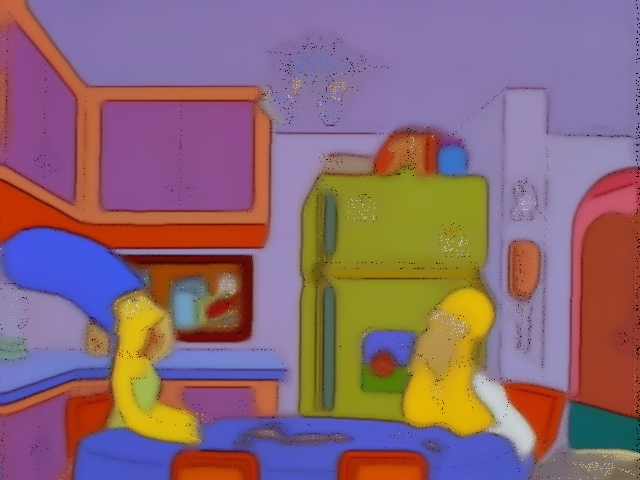}}
		}
		\hskip .5em
		\parbox{.38\figrasterwd}{%
			\subfloat[][10\% random samples]{\includegraphics[width=\hsize]{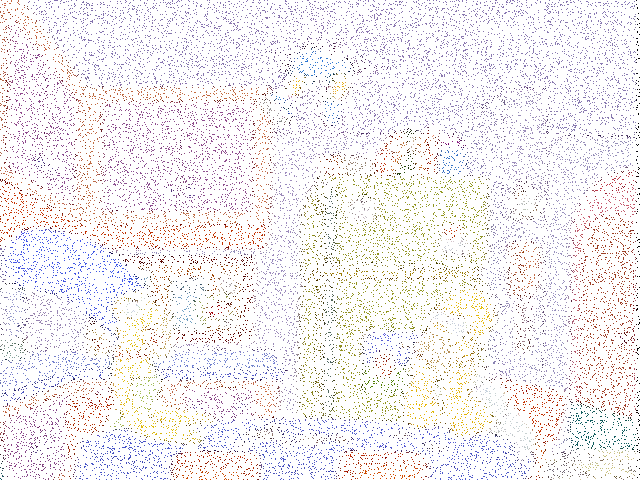}}
			
			\subfloat[][200-Laplacian (PSNR 28.87)]{\includegraphics[width=\hsize]{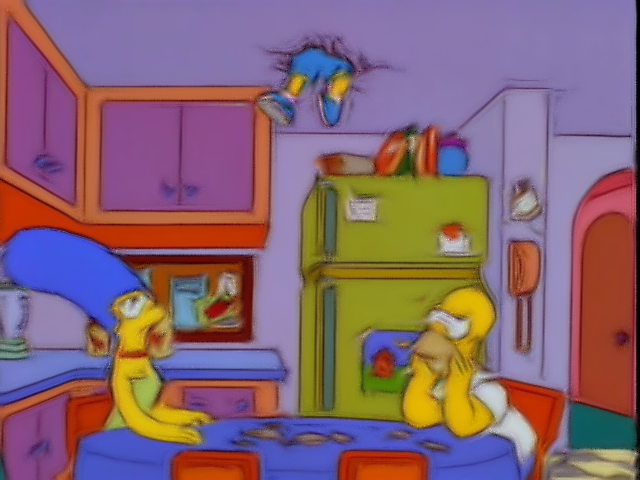}}
		}
	}	
	\caption{Nonlocal inpainting using different extensions.
		(Simpsons image:\href{https://frinkiac.com/img/S08E02/261944.jpg}{https://frinkiac.com/img/S08E02/261944.jpg}).}
	\label{fig:simpsons}
\end{figure}

\paragraph{Nonlocal image inpainting (hole mask).}
In our final example, we apply an inpainting mask with larger holes.
Here, the above approach is modified such that only pixels are compared which are known in both patches.
If less than 10 percent of the pixels in the patches are known, the distance is defined as infinity.
Further, the local part of the distance is omitted and $\sigma$ is chosen as a fixed constant independent of $i$ and $j$.
The algorithm uses the following initialization step:
Iteratively, only pixels with at least one known neighbor in the grid graph are chosen for the inpainting and marked as known pixels afterwards.
This step is repeated until every pixel of the mask is filled, i.e.~every pixel is treated exactly once.
After this initialization, the same iterative procedure as for the random masks is applied and the result for $p=200$ is shown in Fig.~\ref{fig:JoDad}. 
The other parameters are $\sigma = 0.045$, $r=7$, $R=45$, and $K=45$.
Again, we observe quality differences between the 2-Laplacian and 200-Laplacian, especially in the zoomed parts of the image.

\section*{Funding}
This work was supported by the German Research Foundation (DFG)  with\-in the Research Training Group 1932, project area P3.
\begin{figure}[t]
	\centering
	\parbox{\figrasterwd}{
		\centering
		\parbox{.3\figrasterwd}{%
			\subfloat[][Original image]{\includegraphics[width=\hsize]{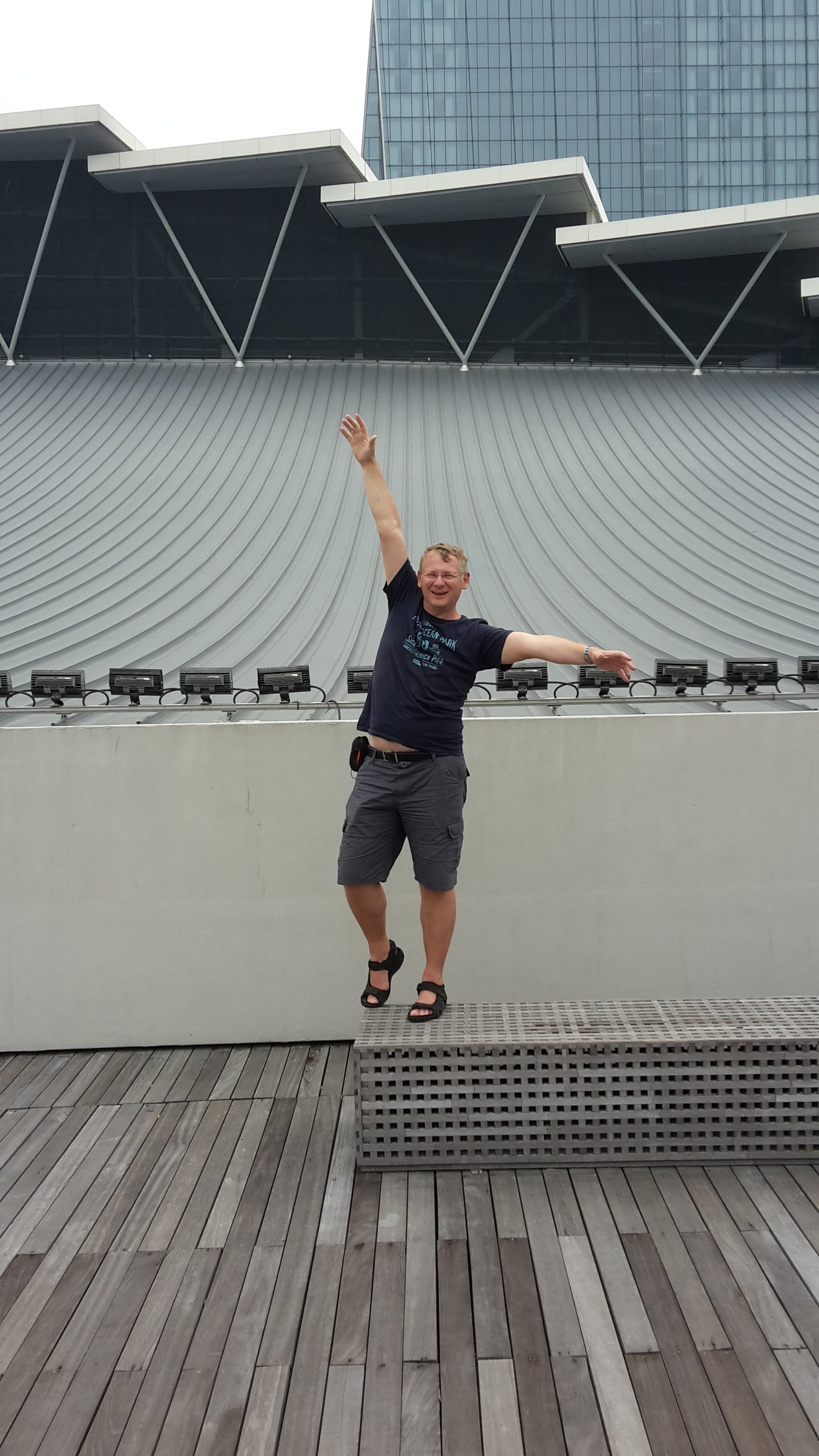}}
			
			\subfloat[][2-Laplacian]{\includegraphics[width=\hsize]{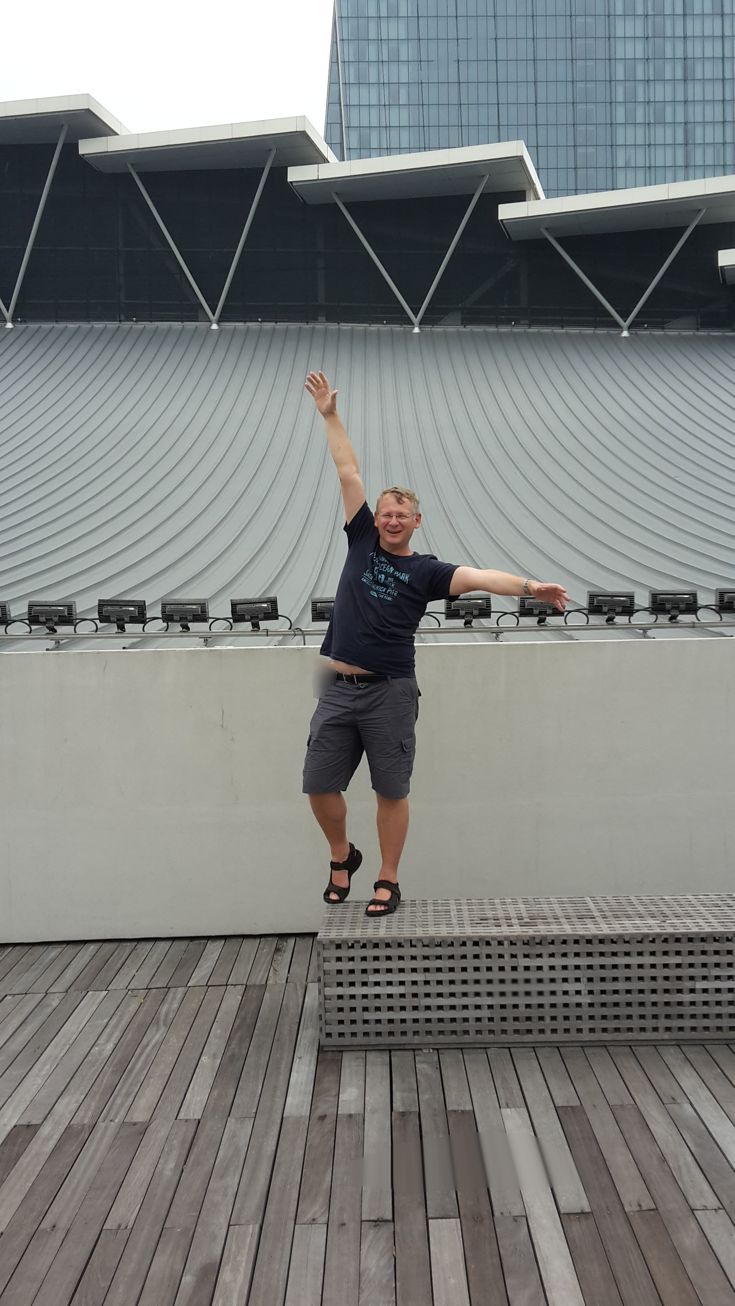}}
		}
		\hskip .5em
		\parbox{.3\figrasterwd}{%
			\subfloat[][Masked image]{\includegraphics[width=\hsize]{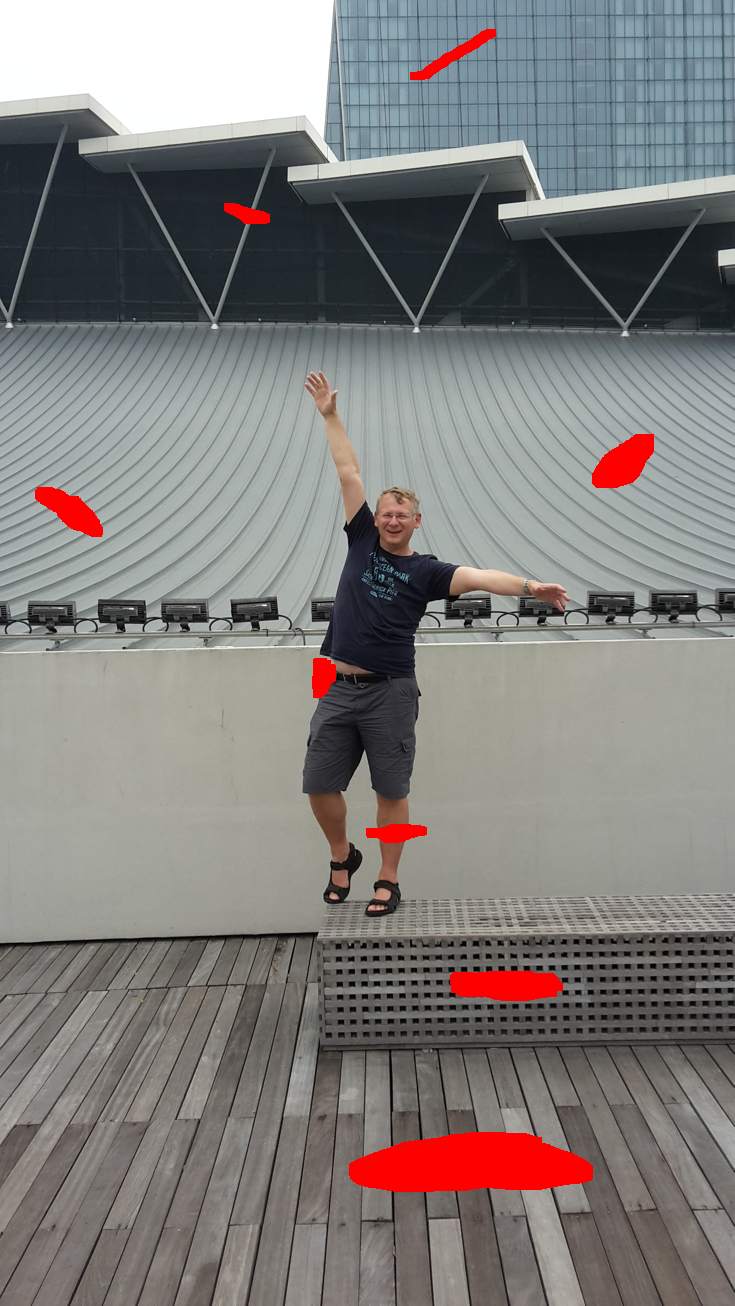}}
			
			\subfloat[][200-Laplacian]{\includegraphics[width=\hsize]{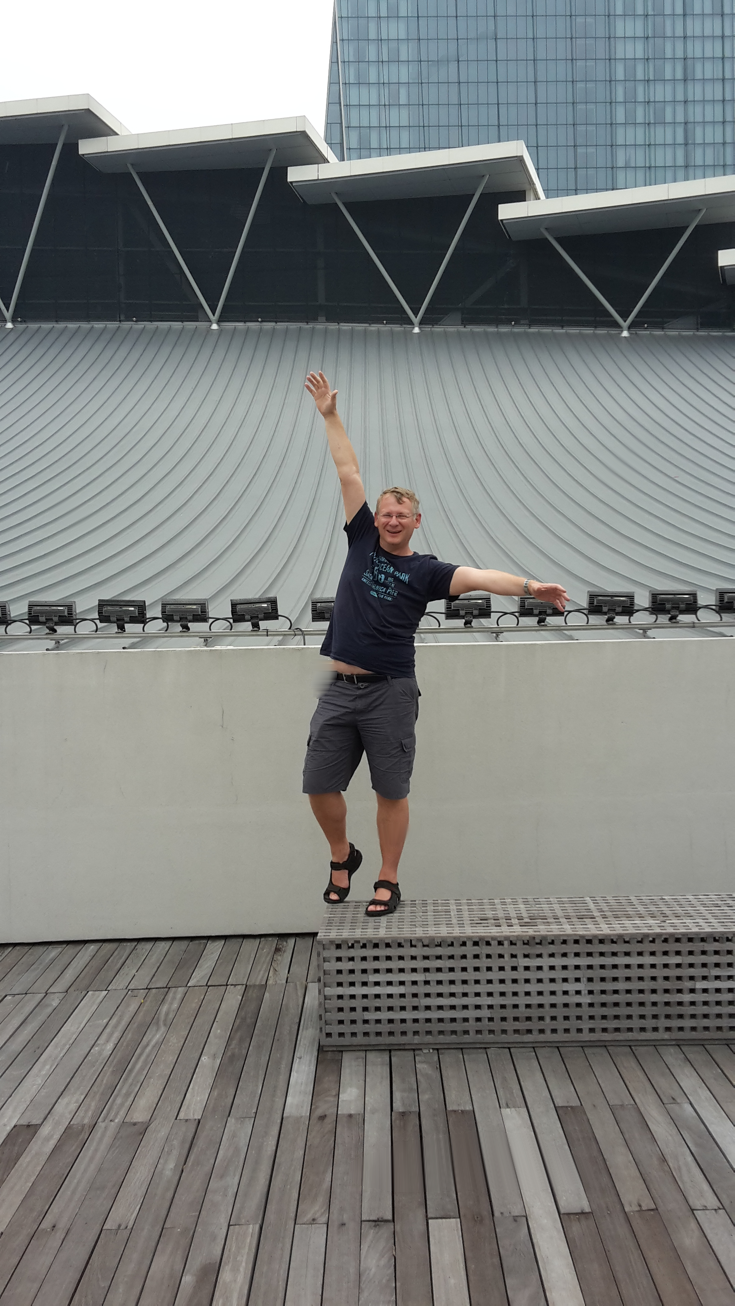}}
		}
		\hskip.5em
		\parbox{.3\figrasterwd}{%
			\subfloat[][Floor 200-Laplacian]{\includegraphics[width=\hsize]{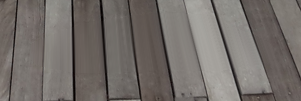}}
			
			\subfloat[][Floor 2-Laplacian]{\includegraphics[width=\hsize]{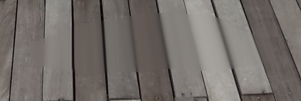}}
			\vskip3em
			\subfloat[][Box 200-Laplacian]{\includegraphics[width=\hsize]{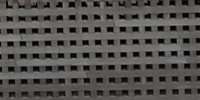}}
			
			\subfloat[][Box 2-Laplacian]{\includegraphics[width=\hsize]{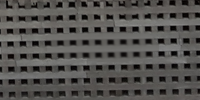}}
		}
	}
	\caption{Non-local inpainting using different extensions.}
	\label{fig:JoDad}
\end{figure}

\bibliographystyle{abbrv}
\bibliography{database}	
\end{document}